\documentclass{USG}

\usepackage[mathscr]{eucal}
\usepackage{mathtools} 

\newcommand{\cO}{\mathcal{O}}

\newcommand{\powres}[3]{\left(\! \frac{\textstyle #1\strut}{\textstyle #2\strut} \!\right)_{\mkern-5mu#3}}
\newcommand{\Q}{\mathbb{Q}}

\newcommand{\ray}[1]{K^{\text{\upshape ray}}(#1)}
\newcommand{\reduce}[1]{\overset{\lower2pt\hbox{$\scriptstyle#1$}}{\Rightarrow}}
\newcommand{\roots}{\boldsymbol\mu}
\newcommand{\row}{\mathscr{R}}
\newcommand{\Z}{\mathbb{Z}}

\newcommand{\mellproot}[1]{\mkern-4mu \root {\raise3pt\hbox{$\scriptscriptstyle m_{\ell}^+$}\mkern-4mu} \of {#1}}
\newcommand{\mellroot}[1]{\mkern-4mu \root {\raise2pt\hbox{$\scriptscriptstyle m_{\ell}$}\mkern-2mu} \of {#1}}

\DeclareMathOperator{\Char}{char}
\DeclareMathOperator{\Gal}{Gal}

\DeclareMathOperator{\SL}{SL}

\newcommand\bigset[2]{\left\{\, #1 
 \mathrel{\left| \vphantom {\left\{ #1 \mid #2 \right\} }
 \right.} #2 \,\right\} }

\renewcommand{\theenumi}{\arabic{enumi}}

\theoremstyle{WBstyleone}
\newtheorem{thm}{Theorem}[section]
\newtheorem{lemma}[thm]{Lemma}
\newtheorem{prop}[thm]{Proposition}
\newtheorem{cor}[thm]{Corollary}

\theoremstyle{WBstylethree}
\newtheorem{definition}[thm]{Definition}
\newtheorem{notation}[thm]{Notation}
\newtheorem*{ack}{Acknowledgments}

\newcounter{stepholder}
\numberwithin{step}{stepholder} 
\renewcommand{\thestep}{\arabic{step}}
\renewenvironment{step}{\refstepcounter{step}\par \medbreak \noindent \textbf{Step \thestep.}\unskip\ \it\ignorespaces}{}

\renewcommand{\thesubsection}{\thesection\Alph{subsection}}

\usepackage{wiley_fix} 

\begin{document}

\title{Bounded elementary generation of \texorpdfstring{$\SL_2$}{SL2}: 
\texorpdfstring{\\}{} 
nearly the end}

\author[1]{B. \`E. Kunyavski\u{\i}}%
\address[1]{%
\orgdiv{Department of Mathematics, }%
\orgname{Bar-Ilan University, }%
\orgaddress{\state{5290002 Ramat Gan, }\country{Israel}}}

\author[2]{D. W. Morris}
\address[2]{%
\orgdiv{Department of Mathematics and Computer Science, }%
\orgname{University of Lethbridge, }%
\orgaddress{\state{Lethbridge, AB, T1K~3M4, }\country{Canada}}}

\author[3]{A. S. Rapinchuk}
\address[3]{%
\orgdiv{Department of Mathematics, }%
\orgname{University of Virginia, }%
\orgaddress{\state{Charlottesville, VA 22904-4137, }\country{USA}}}

\corres{%
B.~\`E.~Kunyavski\u{\i} (\email{kunyav@macs.biu.ac.il})
~|~ D.~W.~Morris  (\email{dave.morris@uleth.ca}) 
~|~ A.~S.~Rapinchuk (\email{asr3x@virginia.edu})%
}

\authormark{Kunyavski\u{I} \textsc{et al.}}
\titlemark{Bounded elementary generation of $\SL_2$}


\dedicated{Dedicated to Alexander Lubotzky on the occasion of his 70th birthday.}

\abstract[ABSTRACT]{Let $\cO_S$ be the ring of $S$-integers of a global field $K$ of any characteristic, where $S$ is a finite set of valuations of $K$ (and $S$ contains all of the archimedean valuations if the characteristic is zero). We prove that if $|S| \ge 2$, then every unimodular $(2\times 2)$-matrix over $\cO_S$ is a product of $\leq 7$ elementary matrices. This nearly optimal bound essentially concludes the investigation
of bounded elementary generation of $\SL_2(\cO_S)$ started over 50 years ago.}

\maketitle

\section{Introduction and statement of the results}\label{S:Intro}

Let $K$ be a global field of arbitrary characteristic, i.e., either an algebraic number field or the function field of a smooth geometrically integral projective curve over a finite field. We let $V^K$ denote the set of all nontrivial valuations of $K$, and for a finite subset $S \subset V^K$ (that contains the set $V^K_{\infty}$ of archimedean valuations if $K$ is a number field), consider the ring of $S$-integers
\begin{equation}\label{E:Sint} 
\cO_S = \{\, a \in K \mid \text{$v(a) \geq 0$ for all $v \in V^K \setminus S$} \,\}.
\end{equation}
An \emph{elementary matrix} in $\SL_2(\cO_S)$ is a matrix of the form 
    {\renewcommand{\arraystretch}{0.5}$\begin{pmatrix} 1 & * \\ 0 & 1 \end{pmatrix}$
or 
    $\begin{pmatrix} 1 & 0 \\ * & 1 \end{pmatrix}$}.
The goal of this paper is to prove the following result that essentially completes a rather long line of research concerning bounded generation of the group $\SL_2(\cO_S)$ by elementary matrices.

\begin{thm}\label{T:1}
Assume that\/ $\vert S \vert \geq 2$. Then every matrix in\/ $\SL_2(\cO_S)$ is a product of at most\/~$7$ elementary matrices.
\end{thm}

We note that the assumption $\vert S \vert \geq 2$ is equivalent to the fact that the group of units $\cO_S^{\times}$ is infinite, and that in the case where $\vert S \vert = 1$, the group $\SL_2(\cO_S)$ is never boundedly generated by elementaries (see \S \ref{S:rank1}).

We will now give a brief, although quite complete, account of the previous work. The investigation into bounded generation of $\SL_2$ by elementaries over rings of $S$-integers in global fields with infinitely many units was opened around 1975 almost simultaneously by Cooke and Weinberger \cite{CW} in characteristic zero and by Queen \cite{Q1,Q2} in positive characteristic. In both situations, the analysis heavily relied on results related to Artin's primitive root conjecture. The difference, however, is that while in characteristic zero Artin's original conjecture has been established in full only assuming the Generalized Riemann Hypothesis (GRH) for certain fields (see Hooley \cite{Hoo}), its analog in positive characteristic is known to hold unconditionally (more precisely, Bilharz \cite{Bil} proved it assuming the Riemann hypothesis for global function fields which was later proved by A.~Weil). The same distinction remains in place for natural variants of Artin's conjecture, and we refer the reader to Lenstra \cite{Len} for a survey of these developments. Using suitable refinements of the results on Artin's conjecture, Cooke and Weinberger \cite{CW} proved that if $\Char K = 0$ and $\vert S \vert \geq 2$, then \emph{assuming a certain form of the (GRH),} one has that every matrix in $\SL_2(\cO_S)$ is a product of $\leq 9$ elementary matrices, while Queen (cf.\ \cite[Theorem 2]{Q2}) showed \emph{unconditionally} that if $\Char K > 0$ and $S$ of size $\geq 2$ contains a point of degree 1 then every element of $\SL_2(\cO_S)$ is a product of $\leq 5$ elementaries.
(See \cite[Remark 8.2]{KPV} for minor corrections to the proof of \cite[Theorem 2]{Q2}.)

The nature of the argument provided by Cooke and Weinberger had led to the question of whether one can establish bounded elementary generation of $\SL_2(\cO_S)$ in characteristic zero \emph{unconditionally.} Liehl \cite{Li} proved this for some special fields $K$. Later, using analytic techniques, Loukanidis and Murty were able to establish the desired result over an arbitrary number field $K$ provided that $S$ is sufficiently large, viz.\ $\vert S \vert \geq \max(5 , 2[K : \Q] - 3)$, cf.\ \cite{LoMu,Mu}.  The first unconditional proof in full generality was given by D.~Carter, G.~Keller and E.~Paige in an unpublished preprint that omitted some details; their argument was completed and made available to the public by D.\,W.~Morris \cite{Mor}. This argument relied on model theory in conjunction with some difficult results of additive number theory dealing with the refinements of the Goldbach problem, and provided no explicit bound on the number of elementaries required. On the other hand, Vsemirnov \cite{Vs} used the results of Heath-Brown \cite{HB} on Artin's primitive root conjecture to implement the approach of Cooke and Weinberger for the ring $\cO_S = \Z[1/p]$ and prove bounded elementary  generation of $\SL_2 \bigl( \Z[1/p] \bigr)$ with a bound of $\leq 5$ elementaries. Finally, Morgan, Rapinchuk and Sury \cite{MRS} proved elementary bounded generation of $\SL_2(\cO_S)$ whenever $\vert S \vert \geq 2$ over an arbitrary number field $K$ with a bound of $\leq 9$ elementaries using only standard results of algebraic number theory such as Artin reciprocity and Chebotarev's Density Theorem (thus, the \emph{same} bound as in \cite{CW} was achieved \emph{without} assuming the (GRH)). Subsequently, the bound was lowered to $\leq 8$ elementaries if $K$ has a real embedding and to $\leq 7$ elementaries assuming a suitable form of the (GRH) by Jordan and Zaytman \cite{JZ} who also showed that the latter bound -- under (GRH) -- can be improved further to $\leq 6$ elementaries if $S$ contains a finite place and even to $\leq 5$ elementaries if $K$ has a real embedding. So, in the general case the bound of $\leq 7$ provided by Theorem~\ref{T:1} \emph{unconditionally} is the same as the bound known earlier assuming the (GRH). Regarding the optimality of this bound, we recall that the example given in \cite{Vs} and corrected in \cite[\S 5]{MRS} demonstrates that at least 5 elementaries are needed to express every element of $\SL_2 \bigl( \Z[1/p] \bigr)$.
(See Subsection~\ref{S:need5} for another proof of this fact.)
Thus, the only remaining question is whether the optimal bound in the general case (without assuming the (GRH)) is $5$, $6$ or $7$.
However, the phrase ``nearly the end'' in the title should not be interpreted as a suggestion that the optimal result will be obtained soon or without substantial additional efforts.

In contrast to the case of characteristic zero, after the work of Queen \cite{Q2}, bounded elementary generation of $\SL_2(\cO_S)$ in positive characteristic has not been investigated further although this property has been recently established for $\SL_n$ $(n \geq 3)$ over the ring of regular functions on any smooth affine curve over a finite field by Trost \cite{Trost} (the case of polynomial rings in one variable was handled earlier by Nica \cite{Nica}) and for arbitrary Chevalley groups of rank $> 1$ over the rings of polynomials and Laurent polynomials in one variable by Kunyavski\u{\i}, Plotkin and Vavilov \cite{KPV}. 
So, here Theorem~\ref{T:1} again provides a nearly final result in the general case. As in the characteristic zero case, the optimal bound is at least 5 in light of the results of \cite{VSS}, see Remark~7.2 in \cite{KPV}.

We also remark that more recent work of Kunyavski\u{\i}, Plotkin and Vavilov both with and without Lavrenov \cite{KLPV,KPV2} quoted an earlier version of Theorem~\ref{T:1} in which the bound on the number of elementary matrices was~8, rather than~7. Therefore, the statements of \cite[Theorem~D]{KLPV} and \cite[Theorem~B]{KPV2} can now be improved by replacing the factor 8 with~7.

The proof of Theorem~\ref{T:1} in characteristic zero was obtained by A.~Morgan, with some assistance from D.\,W.\,Morris. It is a refinement of the argument given in \cite{MRS}, hence relies only on standard results of algebraic number theory and avoids any reference to Artin's conjecture (and of course also to the (GRH)). Normally, A.~Morgan would have been a co-author of the current paper, however our attempts to contact him and get his approval of the text of the paper turned out to be unsuccessful. Under the circumstances, we felt that putting his name on the paper without his consent would be inappropriate, hence finally decided to leave it out. At the same time, we would like to emphasize that he deserves full credit for the result. The proof of Theorem~\ref{T:1} in positive characteristic was obtained by B.~Kunyavski\u{\i} and A.~Rapinchuk. Their argument, just like Queen's, uses results related to Artin's conjecture, however  more general versions of these results discovered by Lenstra \cite{Len}, combined with (a variation of) some results of \cite{CK}, \cite{MRS} and \cite{Trost} are required.

Here is an outline of the paper.  
Section~\ref{S:prelim} sets up notation and recalls basic facts about power residue symbols.
In Section~\ref{S:Opt} we discuss two aspects of the optimality of Theorem \ref{T:1}. First, we show that the assumption $| S | \geq 2$ can never be omitted (this fact was known in characteristic zero but to our knowledge was not mentioned explicitly in positive characteristic). We then discuss the minimal number of elementaries required to write every matrix in $\SL_2(\cO_S)$ as a product when $| S | \geq 2$, and in particular, give a short proof of the known fact that at least five elementaries are needed for $\SL_2 \bigl( \mathbb{Z}[1/p] \bigr)$.
Section~\ref{S:Artin} discusses how Artin's Primitive Root Conjecture is related to results like our main theorem on bounded generation. It also presents a consequence of work of Lenstra related to Artin's Conjecture that will be used in our proof. 
Section~\ref{S:Pos_char} presents the proof of Theorem~\ref{T:1} in the special case where the global field~$K$ has positive characteristic.
Sections \ref{S:setting}--\ref{S:NumberFieldPf} of the paper are devoted to fields of characteristic zero. Section \ref{S:setting} assembles various notions (in addition to those listed in section~\ref{S:prelim}) required to prove Theorem \ref{T:1}, most of which go back to \cite{MRS}. In Section \ref{S:MRS}, we state several results from \cite{MRS} that play a crucial role in the proof of Theorem \ref{T:1}, which is given in Section \ref{S:NumberFieldPf}.

\begin{ack}
We thank Denis Osin and two anonymous referees for helpful comments on a previous version of the paper.
The first and second authors would like to thank the Mathematics Department of the University of Virginia for its hospitality during their visits. 
The second author was partially supported by a research grant from the Natural Sciences and Engineering Research Council of Canada. 
\end{ack}

\section{Notations and properties of the power residue symbol} \label{S:prelim}

\subsection{Notations}\label{SS:Notations} Throughout this paper $K$ will denote a global field. 
We will use $p$ to denote the characteristic of $K$ \emph{if} it is positive. We let $V^K$ denote the set of (equivalence classes of) nontrivial valuations of $K$, with the subsets of archimedean and nonarchimedean valuations denoted $V^K_{\infty}$ and $V^K_f$, respectively (of course, $V^K_{\infty} = \varnothing$ if $\Char K > 0$).  Valuations $v \in V^K_f$ are assumed to be normalized so that their value group is $\Z$. Given a (nonempty) finite subset $S \subset V^K$ containing $V^K_{\infty}$, we let $\cO_S$ denote the corresponding ring of $S$-integers -- cf. Equation~(\ref{E:Sint}) above. Nonzero prime ideals of $\cO_S$ are in a natural bijective correspondence with valuations in $V^K \setminus S$, and the valuation corresponding to a prime ideal $\mathfrak{p}$ of $\cO_S$ will be denoted $v_{\mathfrak{p}}$. 

For a valuation $v \in V^K$, we let $K_v$ denote the corresponding completion. For $v \in V^K_f$, we let $\cO_v$ and $\mathcal{P}_v$  respectively denote the valuation ring in $K_v$ and its maximal ideal. If $v = v_{\mathfrak{p}}$, we will occasionally write $K_{\mathfrak{p}}$ etc. instead of $K_v$ etc.  

As usual, we let $\roots_K$ denote the group of all roots of unity in $K$, and set $\mu = | \roots_K \!|$. Furthermore, for any $m \mid \mu$, we let $\roots_m$ denote the group of $m$-th roots of unity in $K$. For any integer $n$ that is not divisible by $\Char K$ we let $\zeta_n$ denote a primitive $n$-th root of unity in a fixed algebraic closure of $K$. Given an integer $m > 0$ and a prime $\ell$, we let $m_{\ell}$ denote the highest power of $\ell$ that divides $m$, and set $m_{\ell}^+ = \ell \cdot m_{\ell}$. The group of units of a ring $R$ with identity is denoted $R^{\times}$. A unit $u \in R^{\times}$ of infinite order is called \emph{fundamental} if the cyclic group $\langle u \rangle$ is a direct factor of $R^{\times}$. We note that when $| S | > 1$, the group $\cO_S^{\times}$ is an infinite finitely generated abelian group, hence possesses fundamental units. 

Given a finite Galois extension $F/K$ with Galois group $G = \Gal(F/K)$, for a prime $\mathfrak{p}$ of $K$ that is unramified in $F$, we let $(\mathfrak{p} , F/K)$ denote the conjugacy class in $G$ consisting of the Frobenius automorphisms associated with the lifts of $\mathfrak{p}$. 

\subsection{Power residue symbol and its properties}\label{SS:PRS} 

Fix $v \in V^K$. For every $m$ such that $K_v$ contains a primitive $m$-th root of unity one defines the \emph{power residue symbol} of degree $m$: 
$$
\powres{{\star},{\star}}{v}{m} \colon K_v^\times \times K_v^\times \to \roots_m, 
$$
see \cite[Ch. XII]{AT} and \cite[(A.13)]{BMS}. (It should be noted that the definitions in these sources are reciprocal of one another, but this does not affect their properties; we will stick with the definition given in \cite{BMS}.) It is well-known that the power residue symbol is bi-multiplicative and skew-symmetric. If $v = v_{\mathfrak{p}}$ for a prime ideal $\mathfrak{p}$ of $\cO_S$ then we will write $\powres{{\star},{\star}}{\mathfrak{p}}{m}$ instead of $\powres{{\star},{\star}}{v}{m}$. We will use the following properties of the power residue symbol: 
    \begin{enumerate}[label=\upshape{(\roman*)}, ref=\roman*]
    
    \item \label{powres-power}
    \emph{$\powres{a,b}{v}{m} = 1$ if either $a$ or $b$ lies in the group of $m$-th powers $(K_v^{\times})^m$.} This immediately follows from the definition. 

    \item \label{powres-cong}
    \emph{If $v = v_{\mathfrak{p}}$, the degree $m$ is prime to the norm of\/~$\mathfrak{p}$, and $a \in \cO_{\mathfrak{p}}^{\times}$, then for any $b \in K_{\mathfrak{p}}^{\times}$ we have 
$$
\powres{a , b}{\mathfrak{p}}{m} \equiv a^{\,v(b) \,\cdot\, (q-1)/m}  \pmod{\mathcal{P}_{\mathfrak{p}}}, 
$$
where $q = |\cO_S/\mathfrak{p}|$ is the norm of\/~$\mathfrak{p}$.} This is \cite[(A.16)]{BMS}. We observe that reduction modulo $\mathcal{P}_{\mathfrak{p}}$ is injective on $\roots_m$ (cf. Lemma \ref{L:A103}), which, for example, implies that if both $a , b \in \cO_{\mathfrak{p}}^{\times}$ then $\powres{a , b}{\mathfrak{p}}{m} = 1$. 

    \item \label{powres-mpower}
    \emph{If $v = v_{\mathfrak{p}}$ and $m$ is prime to the norm of\/~$\mathfrak{p}$, then for all $a , b \in K_{\mathfrak{p}}^{\times}$ satisfying $v(a) = 1$, $b \in \cO_{\mathfrak{p}}^{\times}$, and $\powres{a , b}{\mathfrak{p}}{m} = 1$, we have $b \equiv c^m \pmod{\mathcal{P}_{\mathfrak{p}}}$ for some $c \in \cO_{\mathfrak{p}}$.} Indeed, it follows from~(\ref{powres-cong}) that in our situation $b^{(q-1)/{m}} \equiv 1 \pmod{\mathcal{P}_{\mathfrak{p}}}$, where $q = | \cO_S/\mathfrak{p} |$. So, the residue of $b$ is in ${\bigl((\cO_S/\mathfrak{p})^{\times}} \bigr){}^m$. 

    \item (Reciprocity Law) \label{powres-reciprocity}
    \emph{For all $a , b \in K^{\times}$, we have $\displaystyle 
\prod_{v \in V^K} \powres{a , b}{v}{\mu} = 1$.} See \cite[Ch. XII, Theorem 13]{AT} and \cite[(A.19)]{BMS}.  

    \item \label{SS:PRS-PrimRoot}
    \emph{If $v$ is not complex \textup(i.e., $K_v \neq \mathbb{C}$\textup) then there exist nonzero $a , b \in \cO_S$ such that $\powres{a , b}{v}{\mu}$ is a primitive $\mu$-th root of unity.} Indeed, since $v$ is not complex, we can find $a_0 , b_0 \in K_v^{\times}$ such that $\zeta = \powres{a_0 , b_0}{v}{\mu}$ is a primitive $\mu$-th root of unity. Using weak approximation in conjunction with the fact that the subgroup $({K_v^{\times}})^{\mu} \subset K_v^{\times}$ is open by the assumption that $\mu$ is prime to $p$, we can find $a \in a_0 ({K_v^{\times}})^{\mu} \cap K^{\times}$ and $b \in b_0 ({K_v^{\times}})^{\mu} \cap K^{\times}$. Furthermore, multiplying $a$ and $b$ by the $\mu$-th powers of appropriate nonzero elements of $\cO_S$, we may assume that $a , b \in \cO_S$. On the other hand, it follows from~(\ref{powres-power}) and the bi-multiplicativity of the symbol that  $\powres{a , b}{v}{\mu} = \zeta$. 

    \item \label{SS:PRS-proots}
    \emph{Suppose that $\Char K = 0$. Let $v \in V^K_f$, let $p$ be the associated prime, and let $e \geq 0$ be such that $p^e \mid \mu_{K_v}$. Then $\powres{\cO_v^{\times} , \cO_v^{\times}}{v}{p^e} = \roots_{p^e}$.} This is a particular case of \cite[(A.17)]{BMS}. 

    \item \label{SS:PRS-takepwer}
    \emph{For $k \mid m$, the symbol $\powres{a,b}{v}{m/k}$ equals the $k$-th power $\powres{a,b}{v}{m}^k$.} See \cite[(A.13)]{BMS}.

    \end{enumerate}

\section{On optimality of Theorem \ref{T:1}}\label{S:Opt}

Before proceeding to the proof of Theorem \ref{T:1} -- first in positive characteristic, and then in characteristic zero -- we would like to discuss two aspects of the optimality of this result. One deals with the restriction $| S | \geq 2$ -- we will prove that this restriction can {\it never} be omitted. Second, under the assumption that $| S | \geq 2$, Theorem \ref{T:1} yields the bound of {\it seven} on the number of elementaries required to write every matrix in $\SL_2(\cO_S)$ as a product. On the other hand, Proposition 5.1 in \cite{MRS} shows that $\SL_2\bigl(\mathbb{Z}[1/p]\bigr)$ for every prime $p > 7$ contains matrices that are {\bf not} products of {\it four} elementaries. This leaves us with only three possibilities, viz.\ 5, 6, or 7, for the minimal number of elementaries needed in the general case. We will put the argument in \cite[\S 5]{MRS} in a more general setting, shorten it  and discuss a connection with the stable range condition. We hope that this discussion will be  helpful for determining the optimal bound in the future.

\subsection{The case \texorpdfstring{$\vert S \vert = 1$}{|S| = 1}}\label{S:rank1}

\begin{prop} \label{P:rank1}
If $\vert S \vert = 1$ then the group $\Gamma = \SL_2(\cO_S)$ is not boundedly generated by elementary matrices.
\end{prop}

\begin{proof}
The critical input needed for the argument is that if $\vert S \vert = 1$ then $\Gamma$ has a subgroup $\Delta$ of finite index that admits a surjective
homomorphism $\varphi \colon \Delta \to F$ onto a nonabelian free group, see \cite{Lub1}. More precisely, in characteristic zero the only possibilities with $\vert S \vert = 1$ are: (1) $\cO_S = \Z$, and (2) $\cO_S$ is the ring of integers of an imaginary quadratic field $K = \Q(\sqrt{-d})$. It is well-known that in the first case, the commutator subgroup $[\Gamma , \Gamma]$ has index 12 in $\Gamma$ and is a free group on two generators, so the required fact trivially holds. In the second case, the required fact was established by Grunewald and Schwermer \cite{GS}. In positive characteristic, the argument was given in \cite{Lub1}. (It should be noted that this result is an immediate consequence of the structure theory of lattices in the rank one groups over local fields of positive characteristic developed in Lubotzky \cite{Lub2} and Baumgartner \cite{Baum}.)

As was observed in \cite{MRS}, in characteristic zero bounded elementary generation of $\Gamma$ would imply its bounded generation (BG) as an abstract group, i.e., the existence of a factorization
$$
\Gamma = \langle \gamma_1 \rangle \cdots \langle \gamma_r \rangle
$$
for some elements $\gamma_1, \ldots , \gamma_r \in \Gamma$ (not necessarily distinct), where $\langle \gamma_i \rangle$ denotes the cyclic subgroup generated by $\gamma_i$. Being a subgroup of finite index, the subgroup $\Delta$ in the above notations would also have (BG), yielding (BG) for the nonabelian free group $F$. It is easy to see, however, that $F$ does not have (BG), providing a contradiction, thus proving the proposition in this case.

Let now $K$ be a field of characteristic $p > 0$. Bounded elementary generation of $\Gamma$ means the existence of a factorization
\begin{equation}\label{E:fact}
\Gamma = V_1 \cdots V_r
\end{equation}
where each $V_i$ coincides with either the subgroup $U^+$ of upper unitriangular matrices, or the subgroup $U^-$ of lower unitriangular matrices. We may assume that the finite-index subgroup $\Delta$ as above is normal in $\Gamma$. Then the normal subgroup $\Omega \subset \Gamma$ generated by the conjugates of $U^+ \cap \Delta$ and $U^- \cap \Delta$ is contained in $\Delta$. It easily follows from Equation~(\ref{E:fact}) that the index $[\Gamma : \Omega]$ does not exceed $t^r$ where 
$$ t = [U^+ : U^+ \cap \Delta] = [U^- : U^- \cap \Delta] \leq [\Gamma : \Delta], $$
hence is finite. In particular, $\varphi(\Omega) \neq \{ e \}$. On the other hand, any $\Gamma$-conjugate of
$U^+ \cap \Delta$ or $U^- \cap \Delta$ is a group of exponent $p$, hence must be mapped by $\varphi$ to $\{ e \}$ since $F$ is torsion-free. It follows that $\varphi(\Omega) = \{ e \}$. A contradiction, completing the proof of the proposition.
\end{proof}

\subsection{At least five elementaries are needed} \label{S:need5}

Recall that a ring $R$ has \emph{stable range~1} if for every $a,b \in R$, such that $aR + bR = R$, there exists $t \in R$, such that $b + at$ is a unit.
It is well known that if $R$ is commutative, then this condition implies $\SL_2(R) = U^- \, U^+ \, U^- \, U^+$ -- cf. Lemma \ref{corner1}. The converse was stated in \cite[Remark 8.2]{KPV} but no proof was given, so we will provide an argument for some special commutative rings -- see Proposition \ref{P:3.1}. This still falls short of proving that $\SL_2(\cO_S)$ contains matrices that are not products of four elementaries even though it is easy to show that $\cO_S$ does not have stable range 1 for any finite $S$. So, we will establish one additional statement that enables one  to give a short proof of \cite[Proposition 5.1]{MRS} that  $\SL_2 \bigl( \Z[1/p] \bigr)$ contains ``non-products'' of four elementaries for any prime $p > 7$. 

\begin{lemma} \label{corner1}
Assume $R$ is a commutative ring, and $(a,b)$ is the first row of $A \in \SL_2(R)$. Then we have
$A \in U^- \, U^+ \, U^- \, U^+$ if and only if there exists $t \in R$, such that $b + at$ is a divisor of $a - 1$.
\end{lemma}

\begin{proof}
$(\Leftarrow)$ Let $t \in R$ be as in the statement. The following transformations of the first row $$(a , b) \, \Rightarrow \, (a , b + at) \, \Rightarrow \, (1 , b + at)$$ can be implemented 
by multiplying $A$ on the right first by a matrix from $U^+$, and then by a matrix from $U^-$. A direct computation shows that 
$$ U^- \, U^+ = \bigset{ \begin{pmatrix} 1 && y \\ x && 1 + xy \end{pmatrix} }{ x,y \in R } . $$
In particular, every element of $\SL_2(R)$ with the first row of the form $(1 , \star)$ belongs to $U^-\, U^+$. Thus, $A \bigr( U^+ \, U^-\bigl) \,  \cap \, U^- \, U^+ \neq \varnothing$, and the required fact follows. 

$(\Rightarrow)$ We have $A \bigr( U^+ \, U^-\bigl) \,  \cap \, U^- \, U^+ \neq \varnothing$, so it follows from the above description of $U^- \, U^+$ that there exist elementaries $e_{12}(t) \in U^+$ and $e_{21}(s) \in U^-$ such that the first row of the matrix $A \, e_{12}(t) \, e_{21}(s)$ is of the form $(1 , \star)$. This amounts to the equation $a + s(b +at) = 1$, showing that $b + at$ divides $a-1$. 
\end{proof}

\begin{prop}\label{P:3.1}
Let $R$ be an integral domain such that the quotient $R/\mathfrak{a}$ is finite for every nonzero ideal~$\mathfrak{a}$ and the group of units $R^{\times}$ is infinite. Then $R$ has stable range~1 if and only if\/ $\SL_2(R) = U^- \, U^+ \, U^- \, U^+$.
\end{prop}

\begin{proof}
($\Rightarrow$) Follows immediately from Lemma~\ref{corner1}.

($\Leftarrow$) Suppose $R$ does not have stable range 1, i.e., there exist $a , b \in R$ such that $aR + bR = R$ but there is no $\lambda \in R$ for which $b + a\lambda$ is a unit. Since the quotient $R/aR$ is finite but the unit group $R^{\times}$ is infinite, there exists $u \in R^{\times}$, $u \neq -1$, satisfying $u \equiv -1 \pmod{aR}$. The finiteness of  $R/(u + 1)R$ implies that the element $(u+1)$ is contained in only finitely many maximal ideals of $R$, and we let $\mathfrak{m}_1, \ldots , \mathfrak{m}_m$ (resp., $\mathfrak{n}_1, \ldots , \mathfrak{n}_n$) denote those maximal ideals that contain (resp., do not contain) $b$. By the Chinese Remainder Theorem, we can find $s \in R$ satisfying 
$$
\text{$s \equiv 1 \pmod{\mathfrak{m}_1 \cdots \mathfrak{m}_m}$ \ and \ $s \equiv 0 \pmod{\mathfrak{n}_1 \cdots \mathfrak{n}_n}$.} 
$$
Set $c = b + as$. Then for each $i = 1, \ldots , m$ we have $c \equiv a \not\equiv 0 \pmod{\mathfrak{m}_i}$ as $b \in \mathfrak{m}_i$ and $a$ and $b$ are relatively prime. Furthermore, for each $j = 1, \ldots , n$ we have $c \equiv b \not\equiv 0 \pmod{\mathfrak{n}_j}$ by construction. Thus, there is no maximal ideal of $R$ that contains both $(u+1)$ and $c$, hence $(u+1)R + cR = R$. 

Let $A \in \SL_2(R)$ be a matrix with the first row $( u + 1 , c)$. To obtain a contradiction, we will show that $A \notin U^- \, U^+ \, U^- \, U^+$. Indeed, otherwise by Lemma \ref{corner1}, there would exist $t \in R$ such that $d \coloneqq c + (u+1)t$ divides $(u+1) - 1 = u$. Since $u$ is a unit, this would imply that $d$ is also a unit. On the other hand, since $u+1 \equiv 0 \pmod{aR}$, we have 
$$
d \equiv c \equiv b \pmod{aR}, 
$$
i.e., $d$ is of the form $b + a\lambda$, contradicting our initial assumption that $R$ has no units of this form. 
\end{proof}

The ring $\cO_S$ with $| S | \geq 2$ satisfies the assumptions of Proposition \ref{P:3.1}, so the fact that $\cO_S$ does not have stable range 1 for any finite $S$ implies that for $\Gamma = \SL_2(\cO_S)$ we have $\Gamma \neq U^- \, U^+ \, U^- \, U^+$. By taking inverses, we obtain that also $\Gamma \neq U^+ \, U^- \, U^+ \, U^-$. However, these two facts still do not allow us to conclude that 
$$
\Gamma \neq U^- \, U^+ \, U^- \, U^+ \bigcup U^+ \, U^- \, U^+ \, U^-, 
$$
which is precisely the statement that $\Gamma$ contains matrices that are not products of four elementaries. Such claims can be proved in some situations by using the following generalization of \cite[Lemma 5.2]{MRS}. 

\begin{lemma} \label{L:1+unit}
Assume that $u$ and~$v$ are units in an integral domain~$R$, and let 
	$$ \mathfrak{a}  = (u + 1)R + (v + 1)R . $$
Any matrix $A \in \SL_2(R)$ of the form
	   $$ A = \begin{bmatrix} u + 1 & b \\ c & v + 1 \end{bmatrix} $$
which is a product of four elementary matrices, satisfies the congruence 
	$$ A \equiv \begin{bmatrix} 0 & w \\ -w^{-1} & 0 \end{bmatrix} \pmod{\mathfrak{a}} 
	\quad $$
\end{lemma}

\begin{proof}
Replacing $A$ with its inverse if necessary, we may assume that $A \in U^- \, U^+ \, U^- \, U^+$. By Lemma~\ref{corner1}  there exists $t \in R$, such that $w \coloneqq b + (u+1)t$ is a divisor of $(u + 1) - 1 = u \in R^{\times}$, and hence is a unit itself. We have $b \equiv w \pmod{\mathfrak{a}}$, and both diagonal entries are in~$\mathfrak{a}$. Therefore $1 = \det A \equiv -w c \pmod{\mathfrak{a}}$, so $c \equiv -w^{-1} \pmod{\mathfrak{a}}$, as required. 
\end{proof}

This lemma provides a short proof of the result we already mentioned above, which was established in~\cite{MRS}  with the additional hypothesis $p > 7$.

\begin{cor}[cf.\ {\cite[Proposition  5.1]{MRS}}]
If $p$ is any rational prime, then not every matrix in $\SL_2 \bigl( \Z[1/p] \bigr)$ is a product of four elementary matrices.
\end{cor}

\begin{proof}
If $p \neq 3$, then $(p^4 + 1)^2 \equiv (1+ 1)^2 \equiv 1 \pmod{3}$, so there is a matrix in $\SL_2(\Z)$ of the form
	 $$ A= \begin{pmatrix} p^4 + 1 & 3 \\ \star & p^4 + 1 \end{pmatrix} . $$
Modulo $p^4 + 1$, every unit in $\Z[1/p]$ is congruent to $\pm p^k$ with $0 \le k \le 3$, and therefore no unit is congruent to~$3$ (because $p^3 + 3 < p^4 + 1$). So the fact that $A$ is not a product of four elementaries follows from Lemma~\ref{L:1+unit}. 

For $p = 3$, let
	$$ A = \begin{pmatrix} -3^3 + 1 & 5 \\ 135 & -3^3 + 1 \end{pmatrix} . $$
Modulo $-3^3 + 1 = -26$, every unit in $\Z[1/3]$ is congruent to $\pm 3^k$ with $0 \le k \le 2$, and therefore no unit is congruent to~$5$ (because $3^2 + 5 < 3^3 - 1$), so we can again use Lemma~\ref{L:1+unit}.
\end{proof}

\section{Bounded elementary generation \texorpdfstring{of $\SL_2(\cO_S)$}{} and Artin's Primitive Root Conjecture}\label{S:Artin}

\subsection{Residues of units and bounded elementary generation.} The following observation goes back  at least to \cite{CW}. 

\begin{lemma}\label{L:A02}
Let $a , b \in \cO_S$ be such that $a\cO_S + b\cO_S = \cO_S$, and let $A \in \SL_2(\cO_S)$ be a matrix with the first row $(a , b)$. 
    \begin{enumerate}

    \item \label{L:A02-3elem}
    If $b \in \cO_S^{\times}$ then $A$ is a product of $\leq 3$ elementaries. 

    \item \label{L:A02-4elem}
    If $b \equiv w \pmod{a}$ for some unit $w \in \cO_S^{\times}$ then $A$ is a product of $\leq 4$ elementaries. 

    \item \label{L:A02-5elem}
    If the residue class $a + bt$ $(t \in \cO_S)$ contains an element $c$ such that $b \equiv w \pmod{c}$ for some unit $w \in \cO_S^{\times}$ then $A$ is a product of $\leq 5$ elementaries.
    \end{enumerate}
\end{lemma}
\begin{proof}
If $b \in \cO_S^{\times}$, then by one elementary row operation $A$ can be transformed to a matrix of the form 
    $\begin{pmatrix} \star & \star \\ \star & 1   \end{pmatrix}$. 
The latter can be transformed by two elementary operations to the identity matrix, proving~(\ref{L:A02-3elem}). Now, if $A$ is as in~(\ref{L:A02-4elem}), then one elementary column operation transforms it to a matrix satisfying the assumption in~(\ref{L:A02-3elem}). Finally, if $A$ satisfies~(\ref{L:A02-5elem}) then one elementary column operation transforms it to a matrix as in~(\ref{L:A02-4elem}).
\end{proof}

For $c \in \cO_S$, $c \neq 0$, we let $\rho_c \colon \cO_S \to \cO_S/c\cO_S$ denote the canonical homomorphism (``reduction modulo $c$''). If $c$ belongs to the residue class $a + bt$ $(t \in \cO_S)$ as above, then $b\cO_S + c\cO_S = \cO_S$, and consequently $\rho_c(b) \in \left( \cO_S/c\cO_S\right)^{\times}$. So, the key assumption in Lemma \ref{L:A02}(\ref{L:A02-5elem}) would be satisfied if we knew that the restriction $\rho_c \colon \cO_S^{\times} \to ( \cO_S/c\cO_S )^{\times}$ is surjective. The question of the existence of a $c$ in the above residue class for which the surjectivity on units  holds is highly nontrivial and relates to the famous Artin's Primitive Root Conjecture (cf.~\cite{Hoo}). A modification of this conjecture that is most relevant to our discussion states that if $\alpha \in  \cO_S$ is not a proper power (in $K^{\times}$) then the set of prime ideals $\mathfrak{p}$ of $\cO_S$ such that for the reduction modulo $\mathfrak{p}$ map $\psi_{\mathfrak{p}} \colon \cO_S \to \cO_S/\mathfrak{p}$, the image $\psi_{\mathfrak{p}}(\alpha)$ is a generator of $(\cO_S/\mathfrak{p})^{\times}$ (i.e., is a ``primitive root'' modulo $\mathfrak{p}$), is \underline{infinite} (in fact, has positive Dirichlet density). We note that if $| S | \geq 2$, then the unit group $\cO_S^\times$ is infinite but finitely generated, hence contains  a unit $\varepsilon \in \cO_S^{\times}$ of infinite order such that the cyclic group $\langle\varepsilon \rangle$ is a direct factor of $\cO_S^{\times}$ (a \emph{fundamental} unit), which is clearly \underline{not} a proper power. So, in this case Artin's Conjecture would yield infinitely many prime ideals $\mathfrak{p}$ such that the restriction $\psi_{\mathfrak{p}} \colon \langle \varepsilon \rangle \to (\cO_S/\mathfrak{p})^{\times}$, and hence also the restriction 
$\psi_{\mathfrak{p}} \colon \cO_S^{\times} \to (\cO_S/\mathfrak{p})^{\times}$, is surjective. While this is \emph{related} to the verification of the key assumption of Lemma~\ref{L:A02}(\ref{L:A02-5elem}), it is still insufficient for the proof of bounded elementary generation of $\SL_2(\cO_S)$, because we need some generator of $\mathfrak{p}$ to be in a certain residue class. A conditional proof of the bounded generation in characteristic zero in \cite{CW} was obtained along these lines through establishing the following result (in a slightly different language) related to Artin's Conjecture under the assumption of a certain form of the Generalized Riemann Hypothesis (GRH). 
\begin{thm}[cf.~{\cite[Theorem 2.6]{CW}}]\label{T:CW-RH}
 Let $K$ be a number field, $\mu$ be the number of roots of unity in $K$, and $S \subset V^K$ be a subset containing $V^K_{\infty}$ of size $\geq 2$  so that the group of units $\cO_S^{\times}$ is infinite. Assume that (GRH) holds as stated in \cite{CW}. 
Then for any $a , b \in \cO_S$ such that $a\cO_S + b\cO_S = \cO_S$, the residue class $a + bt$, $t \in \cO_S$, contains infinitely many  elements $\pi$ with the following properties: 
    \begin{enumerate} 
    
    \item $\mathfrak{p} \coloneqq \pi \cO_S$ is a prime ideal of $\cO_S$; 

    \item the index $[(\cO_S/\mathfrak{p})^{\times} : \psi_{\mathfrak{p}}(\cO_S^{\times})]$ divides $\mu$. 

    \end{enumerate}
\end{thm}

Since the required form of (GRH) remains unproven, this result (as well as the ensuing proof of bounded generation of $\SL_2(\cO_S)$) is still conditional. As we already mentioned in the introduction, an unconditional proof of bounded elementary generation of $\SL_2(\cO_S)$ in characteristic zero was first obtained in unpublished work of D.~Carter, G.~Keller, and E.~Paige by a different method, cf.\ \cite{Mor}. (In Sections \ref{S:setting}-\ref{S:NumberFieldPf} we will refine methods from \cite{MRS} to prove Theorem \ref{T:1} in this case.) However, an analog of Theorem~\ref{T:CW-RH} in positive characteristic  (see Proposition \ref{P:A11}) can be derived from an unconditional result of  H.W.~Lenstra \cite{Len} related to Artin's Conjecture, which in turn relies on the work of H.~Bilharz \cite{Bil} who established Artin's Conjecture in the function field case.

\subsection{A result of Lenstra \texorpdfstring{\cite{Len}}{}} First, let us fix some additional notations assuming that $K$ is a global field of characteristic $p > 0$. We let $\mathbb{F}$ denote the algebraic closure of the prime field $\mathbb{F}_p$ in $K$, so $\mathbb{F}$ is a finite field. (Then $K$ is isomorphic to the function field $\mathbb{F}(C)$ of a geometrically integral smooth projective curve $C$ defined over $\mathbb{F}$, and $K$ can also be realized as a finite extension of the field $\mathbb{F}(T)$ of rational functions.) Letting $q = |\mathbb{F}|$, we have that $\mu \coloneqq q - 1$ is the number of roots of unity in $K$. 

As above, let $S \subset V^K$ be a finite subset, and let $W \subset \cO_S^{\times}$ be a subgroup (automatically finitely generated). Given a finite Galois extension $F/K$ with Galois group $G = \Gal(F/K)$, fix a conjugacy class $C$ in $G$ and an integer $m > 0$ prime to $p$.

\begin{definition}\label{D:1}
Let $M = M(F/K, C, W, m)$  denote the set of prime ideals $\mathfrak{p}$ of $\cO_S$ that are unramified in $F$ and satisfy 
    \begin{itemize}
    
    \item $(\mathfrak{p} , F/K) = C$, 
    
    \item the index $[(\cO_S/\mathfrak{p})^{\times} : \psi_{\mathfrak{p}}(W)]$ divides $m$. 
    
    \end{itemize}
\end{definition}

Given $m > 0$ prime to $p$ and a prime $\ell \neq p$, we let 
$$
L_{W, m, \ell} = K(\zeta_{m_{\ell}^+}, \mellproot{W})
\text{\quad where\quad}
\mellproot{W} = \bigset{ \mellproot{w} }{ w \in W }
.
$$

We refer the reader to \cite[Ch. VII, \S13]{Neukirch} and \cite[\S4]{Ros2}  for the definition of the \emph{Dirichlet density} $d(M)$ of a set of primes $M$ of $K$; in fact, all we need in the current paper is that finite sets have Dirichlet density zero. We will now state a result of Lenstra \cite{Len}. 

\begin{thm}[{\cite[Theorem 4.1]{Len}}]\label{T:Lenstra2} 
Let $h$ be the product of those primes $\ell \neq p$ for which $W \subset ({K^{\times}})^{m_{\ell}^+}$. Then the following conditions are equivalent: 

    \begin{enumerate}

    \item \label{T:Lenstra2-density}
    $d(M) > 0$, 

    \item \label{T:Lenstra2-sigma}
    there exists $\sigma \in \Gal(F(\zeta_h)/K)$ whose restrictions satisfy the conditions  
$$ \text{$\sigma \vert F \in C$ \qquad  and \qquad  $\sigma \vert L_{W, m, \ell} \neq \mathrm{id}_{L_{W, m, \ell}}$ \  for every \  $\ell$ \  with \  $L_{W, m, \ell} \subset F(\zeta_h)$.}
$$
    \end{enumerate}
\end{thm}

We now specialize this to the situation where  $| S | \geq 2$, hence the group of units $\cO_S^{\times}$ is infinite. 

\begin{prop}\label{P:A11}
Let $F/K$ be a finite abelian extension, and fix an automorphism $\sigma \in \Gal(F/K)$. If $|S| \ge 2$, then $\cO_S$ has infinitely many prime ideals $\mathfrak{p}$ that satisfy the following two conditions: 

    \begin{enumerate}
    
    \item $(\mathfrak{p} , F/K) = \sigma$, 

    \item $[{(\cO_S/\mathfrak{p})}^{\times} : \psi_{\mathfrak{p}}(\cO_S^{\times})]$ divides $\mu = q-1$.

    \end{enumerate}
\end{prop}
\begin{proof}
We will apply Theorem \ref{T:Lenstra2} with $W = \cO_S^{\times}$ and $m = \mu$. Since 
$W$ is infinite but finitely generated, there exists 
$w \in W$ of infinite order such that the cyclic group $\langle w \rangle$ is a direct factor of $W$ (i.e., $w$ is a \emph{fundamental} $S$-unit). Then, in particular, $w \notin \roots_K \cdot (K^{\times})^{\ell}$ for all primes $\ell \neq p$, and hence we have $h = 1$ and $F(\zeta_h) = F$ in the notations of Theorem \ref{T:Lenstra2}. Furthermore,  it follows from \cite[Proposition 2.1]{MRS} that $K( \mellroot{w})$ is the maximal abelian over $K$ subextension of $K \bigl( \mellproot{w} \bigr)$. Since 
$[K(\mellproot{w}) : K] = m_{\ell}^+$ and $[K(\mellroot{w}) : K] = m_{\ell}$ (cf.\ \cite[Corollary 2.4(i)]{MRS}) and $F/K$ is abelian, we obtain that $\mellproot{w} \notin F$. Since $\mellproot{w} \in L_{W, m, \ell}$, this implies $
L_{W, m, \ell} \not\subset F$ 
for all $\ell \neq p$. Thus, for $C = \{\sigma\}$ the condition~(\ref{T:Lenstra2-sigma}) of Theorem \ref{T:Lenstra2} is satisfied, and therefore $M = M(F/K, C, W, m)$ has positive Dirichlet density, hence is infinite. On the other hand, every prime $\mathfrak{p} \in M$ satisfies the requirements of the proposition (cf.\ Definition \ref{D:1}).     
\end{proof}

\section{Proof of Theorem \ref{T:1} in positive characteristic}\label{S:Pos_char}

\subsection{Special primes in a given residue class} First, we establish the following unconditional analog of Theorem \ref{T:CW-RH} in positive characteristic. 

\begin{prop}\label{P:A01} 
Let nonzero $a , b \in \cO_S$ be such that $a\cO_S + b\cO_S = \cO_S$. 
Furthermore, suppose that for each $v \in S$ we are given an element $x_v \in K_v^{\times}$ and an open subgroup $V_v \subset K_v^{\times}$ of finite index. 
Then the residue class $a + bt$, $t \in \cO_S$, contains infinitely many elements 
$\pi$ such that 

    \begin{enumerate}
    \item $\mathfrak{p} \coloneqq \pi \cO_S$ is a prime ideal of $\cO_S$;

    \item $[(\cO_S/\mathfrak{p})^{\times} : \psi_{\mathfrak{p}}(\cO_S^{\times})]$ divides $\mu$;

    \item $\pi \in x_vV_v$ for all $v \in S$. 
    \end{enumerate}
\end{prop}

The proposition follows from Proposition \ref{P:A11} and the following statement that relates elements in a residue class to the Frobenius automorphisms in an appropriate abelian extension. (We note that this relation is classical and was used, for example, in the proof of Dirichlet's Theorem \cite[(A.12)]{BMS} and also in \cite{Q2}.)

\begin{prop}\label{P:4} 
Let nonzero $a , b \in \cO_S$ be such that 
\begin{align*}
a\cO_S + b\cO_S = \cO_S, 
\end{align*}
and set $S_b = \{ v \in V^K \setminus S \ | \ v(b) \neq 0 \}$. 
Furthermore, suppose that for each $v \in S$ we are given an element $x_v \in K_v^{\times}$ and an open subgroup $V_v \subset K_v^{\times}$ of finite index. Then there exist a finite abelian Galois extension $F/K$ and an automorphism $\sigma \in \Gal(F/K)$ such that any prime ideal $\mathfrak{p} \notin S_b$  of $\cO_S$ which  is unramified in $F$ and has the property that $(\mathfrak{p} , F/K) = \sigma$ is principal and has a generator $\pi$ satisfying $\pi \equiv a \pmod{b \cO_S}$ and $\pi \in x_vV_v$ for all $v \in S$. 
\end{prop}
\begin{proof}[Proof \rm
(cf. the proof of Dirichlet's Theorem in \cite{BMS})]
Let $J = J_K$ be the group of ideles of $K$. 
For $v \in V^K \setminus S$ set 
$$
V_v = \{\, u \in \cO_v^\times \ | \ v(u - 1) \geq v(b) \,\}.
$$
Clearly, $V_v = \cO_v^\times$ for all $v \in V^K \setminus S$ lying outside the finite set $S_b$, so the subgroup $\Omega = \prod_{v \in V^K} V_v$ is open in $J$. On the other hand, the quotient 
$$
J / \left( \prod_{v \in S} K_v^{\times} \times \prod_{v \in V^K \setminus S} \cO_v^\times   \right) K^{\times} 
$$
is isomorphic to the class group of $\cO_S$, hence is finite, and the index $[K_v^{\times} : V_v]$ is finite for all $v \in S$, so the index $[J : \Omega K^{\times}]$ is also finite. By the Existence Theorem of Global Class Field Theory \cite[Ch. VIII]{AT}, there exists a finite abelian extension $F/K$ satisfying $N_{F/K}(J_F)K^{\times} = \Omega K^{\times}$, and then the corresponding reciprocity map $\theta_{F/K} \colon J \to \Gal(F/K)$ is surjective with the kernel $\Omega K^{\times}$. Define two ideles $\bar{a} = (\bar{a}_v)$ and $\bar{x} = (\bar{x}_v)$ by: 
    $$
    \bar{a}_v = 
        \begin{cases}
           \ a & \text{if $v \in S_b$}, \\
           \ 1 & \text{otherwise}
          \end{cases}
    \qquad \text{and} \qquad
    \bar{x}_v = 
        \begin{cases}
            \ x_v   &  \text{if $v \in S$}, \\
            \ 1  & \text{otherwise} 
        . \end{cases}
    $$
Set $\sigma = \theta_{F/K}(\bar{a}^{-1} \bar{x}^{-1}) \in \Gal(F/K)$. Now, let $\mathfrak{p} \notin  S_b$ be a prime ideal of $\cO_S$ such that $(\mathfrak{p} , F/K) = \sigma$. Define an idele $r = (r_v)$ by $r_v = 1$ for all $v \neq v_{\mathfrak{p}}$ and $v_{\mathfrak{p}}(r_{\mathfrak{p}}) = 1$. Then 
$$
\theta_{F/K}(r) = (\mathfrak{p} , F/K) = \sigma = \theta_{F/K}(\bar{a}^{-1} \bar{x}^{-1}).
$$
So, $\bar{a} \bar{x} r \in \Omega K^{\times}$, and therefore one can write 
\begin{equation}\label{E:A007}
  r \bar{a} \omega = \bar{x}^{-1} \pi \ \ \text{with} \ \ \omega \in \Omega \ \ \text{and} \ \ \pi \in K^{\times}.   
\end{equation}
Since $v(a) = 0$ for all $v \in S_b$, we conclude from (\ref{E:A007}) that $v(\pi) = 0$ for all $v \in V^K \setminus (S \cup \{ v_{\mathfrak{p}} \})$ and $v_{\mathfrak{p}}(\pi) = 1$, which means that $\pi \in \cO_S$ and $\mathfrak{p} = \pi\cO_S$. Considering in (\ref{E:A007}) the components corresponding to $v \in S$, we obtain that $\pi \in x_v V_v$ for all $v \in S$. Finally, for any $v \in S_b$ we have that $\pi = a u$ with $u$ satisfying $v(u - 1) \geq v(b)$. Then 
$$
v(\pi - a) = v(a(u - 1)) \geq v(b), 
$$
implying that $\pi \equiv a \pmod{b\cO_S}$.     
\end{proof}

\subsection{Key lemma} We will now apply Proposition \ref{P:A01} to prove the following statement that will enable us to transform a general matrix $A \in \SL_2(\cO_S)$ into a matrix that can be handled using  Lemma \ref{L:A02}.   
\begin{lemma}\label{L:A71}
Any matrix $A \in \SL_2(\cO_S)$ can be transformed by at most three elementary operations into a matrix with the first row $(a , b^{\mu})$ where $a , b \in \cO_S$ and the ideal $\mathfrak{p}\coloneqq a\cO_S$ is prime with the index $[(\cO_S/\mathfrak{p})^{\times} : \psi_{\mathfrak{p}}(\cO_S^{\times})]$ dividing $\mu$.    
\end{lemma}
\begin{proof}
We will imitate the argument used to prove \cite[Lemma 3]{CK} and \cite[Lemma 3.1]{Trost} but will also incorporate our Proposition \ref{P:A01}. Fix $v_0 \in S$, and let $(a_1 , b_1)$ be the first row of $A$. We will assume (as we may) that both $a_1$ and $b_1$ are nonzero. 

\refstepcounter{stepholder} 
\begin{step} \label{L:A71Pf-b2}
   The residue class $b_1 + a_1t$ $(t \in \cO_S)$ contains an element $b_2$ such that for some nonzero $c \in \cO_S$ we have
   \begin{enumerate}
    \item the ideal $\mathfrak{q} = b_2\cO_S$ is prime;

    \item \label{L:A71Pf-b2-power}
    $b_2 \in (K_v^{\times})^{\mu}$ for all $v \in S \setminus \{ v_0 \}$; 

    \item \label{L:A71Pf-product}
    $\powres{a_1 , b_2}{\mathfrak{q}}{\mu} \cdot \powres{c , b_2}{v_0}{\mu} = 1$. 
    \end{enumerate}
\end{step}
Indeed, by property \ref{SS:PRS}(\ref{SS:PRS-PrimRoot}) there exist nonzero $x , y \in \cO_S$ for which $\displaystyle \zeta \coloneqq \powres{x , y}{v_0}{\mu}$ is a primitive $\mu$-th root of unity. Then by Dirichlet's Theorem \cite[(A.10)]{BMS} or by our Proposition \ref{P:A01}, the residue class $b_1 + a_1 t$ $(t \in \cO_S)$ contains $b_2$ such that $\mathfrak{q} = b_2 \cO_S$ is a prime ideal and 
$$ \text{$b_2 \in (K_v^{\times})^{\mu}$ for all $v \in S \setminus \{ v_0 \}$,   
\ and \ $b_2 \in y \, (K_{v_0}^{\times})^{\mu}$.}   
$$
(We note that since $\mu$ is prime to $p$, the subgroup of $\mu$-th powers $(K_v^{\times})^{\mu}$ is a finite index open subgroup in $K_v^{\times}$ for any valuation $v$ of $K$.) Then 
$$
\powres{x , b_2}{v_0}{\mu}  = \powres{x , y}{v_0}{\mu}  = \zeta.  
$$
Pick a positive integer $k$ so that $\powres{a_1 , b_2}{\mathfrak{q}}{\mu}^{-1}   = \zeta^k$. Then for $c = x^k$ we have 
$$
\powres{a_1 , b_2}{\mathfrak{q}}{\mu}  \cdot \powres{c , b_2}{v_0}{\mu} = \zeta^{-k} \cdot \zeta^k = 1, 
$$
as required.

\begin{step} \label{L:A71Pf-a2}
The residue class $a_1 + b_2t$ $(t \in \cO_S)$ contains an element $a_2$ such that 
   \begin{enumerate}
    \item \label{L:A71Pf-prime}
    the ideal $\mathfrak{p} = a_2\cO_S$ is prime; 

    \item \label{L:A71Pf-coset}
    $a_2 \in c \, (K_{v_0}^{\times})^{\mu}$;  

    \item \label{L:A71Pf-index}
    the index $[(\cO_S/\mathfrak{p})^{\times} : \psi_{\mathfrak{p}}(\cO_S^{\times})]$ divides $\mu$; 

    \item \label{L:A71Pf-triv}
    $\powres{a_2 , b_2}{\mathfrak{p}}{\mu} = 1$. 

    \end{enumerate}
\end{step}
Indeed, by Proposition \ref{P:A01} the residue class contains an element $a_2$ that satisfies the conditions (\ref{L:A71Pf-prime})--(\ref{L:A71Pf-index}). So, let us verify condition~(\ref{L:A71Pf-triv}). First, let us show that $\powres{a_2 , b_2}{v}{\mu} = 1$ for all $v \in V^K$   different from $v_0, v_{\mathfrak{p}}$ 
and $v_{\mathfrak{q}}$.  If $v = v_{\mathfrak{r}}$ for a prime ideal $\mathfrak{r}$ of $\cO_S$ different from $\mathfrak{p}$ and $\mathfrak{q}$  then the required conclusion follows from the fact that both $a_2 , b_2 \in \cO_{\mathfrak{r}}^{\times}$ (cf.\ property 
\ref{SS:PRS}(\ref{powres-cong})). On the other hand, if $v \in S \setminus \{ v_0 \}$, then according to the condition~(\ref{L:A71Pf-b2-power}) in Step~\ref{L:A71Pf-b2}, we have $b_2 \in (K_v^{\times})^{\mu}$,  and the required fact is immediate (cf. property \ref{SS:PRS}(\ref{powres-power})).

Now, it follows from the Power Reciprocity Law (cf.\ property \ref{SS:PRS}(\ref{powres-reciprocity})) that 
\begin{equation}\label{E:A001}
\powres{a_2 , b_2}{\mathfrak{p}}{\mu} \cdot \powres{a_2 , b_2}{\mathfrak{q}}{\mu} \cdot \powres{a_2 , b_2}{v_0}{\mu} = 1.   
\end{equation} 
Due to condition (\ref{L:A71Pf-coset}) in Step~\ref{L:A71Pf-a2}, we have that $\powres{a_2 , b_2}{v_0}{\mu} = \powres{c , b_2}{v_0}{\mu}$. Furthermore, since $\mathfrak{q} = b_2\cO_S$ and $a_2 \equiv a_1 \pmod{\mathfrak{q}}$,  we have 
$\powres{a_2 , b_2}{\mathfrak{q}}{\mu} = \powres{a_1 , b_2}{\mathfrak{q}}{\mu}$ (cf. property 
\ref{SS:PRS}(\ref{powres-cong})). So, comparing Equation~(\ref{E:A001}) with (\ref{L:A71Pf-product}) in Step \ref{L:A71Pf-b2}, we obtain~(\ref{L:A71Pf-triv}).

\begin{step}
It follows from Steps \ref{L:A71Pf-b2} and \ref{L:A71Pf-a2} that any matrix $A \in \SL_2(\cO_S)$ with the first row $(a_1 , b_1)$ can be transformed by two elementary operations into a matrix $A'$ with the first row $(a_2 , b_2)$ such that for $a = a_2$ the following conditions are satisfied: 

   \begin{enumerate}

    \item $\mathfrak{p} = a \cO_S$ is a prime ideal; 

    \item the index $[(\cO_S/\mathfrak{p})^{\times} : \psi_{\mathfrak{p}}(\cO_S^{\times})]$ divides $\mu$; 

    \item \label{L:A71Pf-powres=1}
    $\powres{a_2 , b_2}{\mathfrak{p}}{\mu} = 1$.

    \end{enumerate}
\end{step}
Since $a$ is a generator of $\mathfrak{p}$ and $b_2$ is a unit modulo $\mathfrak{p}$, condition (\ref{L:A71Pf-powres=1}) implies that $b_2 \equiv b^{\mu} \pmod{\mathfrak{p}}$ for some $b \in \cO_S$ (cf.\ property \ref{SS:PRS}(\ref{powres-mpower})). Then applying one more elementary operation to $A'$, we will obtain a matrix with the first row $(a , b^{\mu})$, as required. 
\end{proof}

\subsection{Conclusion of the proof} Let $A \in \SL_2(\cO_S)$. By Lemma \ref{L:A71}, using at most three elementary operations, we can transform $A$ into a matrix $A'$ with the first row $(a , b^{\mu})$ as in the lemma. Since for $\mathfrak{p} = a\cO_S$, the index $[(\cO_S/\mathfrak{p})^{\times} : \psi_{\mathfrak{p}}(\cO_S^{\times})]$ divides $\mu$, there exists $w \in \cO_S^{\times}$ such that $b^{\mu} \equiv w \pmod{a}$, and then $A'$ is a product of $\leq 4$ elementaries by Lemma \ref{L:A02}(\ref{L:A02-4elem}). Then $A$ is a product of $\leq 7$ elementaries, as required.

\section{Setting the stage for the proof of Theorem \ref{T:1} in characteristic zero} \label{S:setting}

In the remainder of this article, $K$ will be a global field of \emph{characteristic zero} (in other words, a \emph{number field}), and hence $p$ will no longer be used to denote the characteristic but will rather denote a rational prime.

\subsection{Additional notations and definitions}

While we will continue to use the standard notations introduced in subsection \ref{SS:Notations}, 
consideration of the number field case will require some additional notations and notions that are consistent with those used in \cite{MRS}. Here is a list of most of the commonly used notations: 

\begin{itemize}
    \item $\cO_K$ will denote the ring of algebraic integers in $K$ (in other words, the integral closure of $\Z$ in $K$).

    \item For $a \in K^{\times}$, we set $V(a) = \{\, v \in V^K_f \ | \ v(a) \neq 0 \,\}$. (This is a finite set.) 

    \item For an ideal $I$ of a commutative ring $R$ such that the quotient $R/I$ is finite, we let $\phi(I)$ denote the order of the multiplicative group $(R/I)^{\times}$. 

     \item $\row(\cO_S) = \{\, (a , b) \ | \ a\cO_S + b\cO_S = \cO_S \,\}$. (We note that $\row(\cO_S)$ is precisely the set of  first rows of all matrices $A \in \SL_2(\cO_S)$.)
    
\end{itemize}

\medbreak\noindent 
Given two matrices $A , A'$ over a commutative ring $R$ of the same size, we write  $A \reduce{m} A'$ if $A'$ can be obtained from $A$ by $\leq m$ elementary row and column operations over $R$. We note that if $A \in \SL_2(R)$ has the first row $(a , b)$ and $(a , b) \reduce{m} (1 , 0)$ then $A$ is a product of $\leq m + 1$ elementary matrices. Thus, to prove  Theorem \ref{T:1}, it is enough to show that for any $(a , b) \in \row(\cO_S)$ we have $(a , b) \reduce{6} (1 , 0)$. Furthermore, if $(a , b)$ and $(a' , b) \in R^2$ and $a' \equiv a \pmod{b}$ then $(a , b) \reduce{1} (a' , b)$, and symmetrically, if $(a , b') \in R^2$ and $b' \equiv b \pmod{aR}$ then $(a , b) \reduce{1} (a , b')$. These and other simple properties of the relations $\reduce{m}$ can be found in \cite[Lemma 4.1]{MRS}. 

\medbreak 

Let $\mathfrak{p}$ be a nonzero prime ideal of $\cO_K$, and $p$ be the corresponding rational prime. We say that $\mathfrak{p}$ is $\Q$-\emph{split} if $K_{\mathfrak{p}} = \Q_p$ and $p > 2$. 

Given a prime $p > 2$, we consider the group of $p$-adic units $\mathbb{U}_p = \Z_p^{\times}$ and its filtration by congruence subgroups
$$
\text{$\mathbb{U}_p^{(i)} = 1 + p^i \Z_p$ for $i \geq 1$.} 
$$
It is well-known that 
$$
\mathbb{U}_p = \roots_{\Q_p} \times \mathbb{U}_p^{(1)}. 
$$
Furthermore, for any unit $u \in \mathbb{U}_p \setminus \roots_{\Q_p}$, the closure $\overline{\langle u \rangle}$ of the cyclic subgroup generated by $u$ is of the form 
$$
\overline{\langle u \rangle} = \roots' \times \mathbb{U}_p^{(\ell)}
$$
for some subgroup $\roots' \subset \roots_{\Q_p}$ and some integer $\ell \geq 1$ which will be called the $p$-\emph{level} of $u$ and denoted $\ell_{\!p}(u)$. 
\begin{definition}\label{D:level}
Let $\mathfrak{p}$ be a $\Q$-split prime of $\cO_K$, and $p$ be the corresponding rational prime. Then the unit group $\cO_{\mathfrak{p}}^{\times}$ can be naturally identified with $\mathbb{U}_p$, and for $u \in \cO_{\mathfrak{p}}^{\times}$ we define the $\mathfrak{p}$-level $\ell_{\!\mathfrak{p}}(u)$ of $u$ to be the $p$-level of the element of $\mathbb{U}_p$ corresponding to $u$ under the above identification.  
\end{definition}

\subsection{A few lemmas}\label{SS:3lemmas}

\begin{lemma}[{\cite[Lemma 3.1]{MRS}}] \label{MRS3.1}
Let $\mathfrak{p}$ be a $\Q$-split nonzero prime ideal in $\cO_S$, and let $c \in \cO_S$. Then:
	\begin{enumerate}[label=\upshape{(\alph*)}, ref=\alph*]
	\item \label{MRS3.1-cyclic}
	$(\cO_S / \mathfrak{p}^n)^\times$ is cyclic for every $n \ge 1$;
	and
	\item \label{MRS3.1-gen}
	if $c \pmod{\mathfrak{p}^2}$ generates $(\cO_S / \mathfrak{p}^2)^\times$, then $c \pmod{\mathfrak{p}^n}$ generates $(\cO_S / \mathfrak{p}^n)^\times$ for every $n \ge 2$.
	\end{enumerate}
\end{lemma}

\begin{lemma}[{\cite[Lemma 4.5]{MRS}}]\label{L:A101}
Let $\mathfrak{p}$ be a nonzero principal $\Q$-split prime ideal of $\cO_S$ with a generator $\pi$, and let $u \in \cO_S^{\times}$ be a unit of infinite order. Set $\ell = \ell_{\!\mathfrak{p}}(u)$, and let $\lambda$ and $m$ be integers satisfying $\lambda \mid \phi(\mathfrak{p})$ and $m \equiv 0 \pmod{\phi(\mathfrak{p}^{\ell}) / \lambda}$. Given an integer $\delta > 0$ dividing $\lambda$ and an element $b \in \cO_S$ which is prime to $\pi$ and is a $\delta$-th power modulo $\mathfrak{p}$, if $\lambda/\delta$ divides the order of the image of $u$ in $\cO_S/\mathfrak{p}$, then for every integer $t >  0$ there exists an integer $n$ such that 
$$
(\pi^t , b^m) \reduce{1} (\pi^t , u^n). 
$$
\end{lemma}

\begin{lemma}\label{L:A102} 
Every $b \in \cO_S$ can be written in the form $\beta s^{-1}$ with $\beta \in \cO_K$ and $s \in \cO_S^{\times} \cap \cO_K$. If $| S | \geq 2$ then $\cO_S$ has a~fundamental unit $u \in \cO_S^{\times} \cap \cO_K$. 
\end{lemma}
\begin{proof}
For the first assertion, there is nothing to prove if $S = V^K_{\infty}$. Otherwise, we define an ideal $\mathfrak{s}$ of $\cO_K$ by 
$$
\mathfrak{s} = \prod_{v \in S \setminus V^K_{\infty}} \mathfrak{p}_v
$$
where $\mathfrak{p}_v$ is the prime ideal of $\cO_K$ corresponding to $v \in V^K_f$. Using the finiteness of the class number of $\cO_K$, we can choose $d > 0$ so that $b\mathfrak{s}^d \subset \cO_K$ and $\mathfrak{s}^d$ is principal. If $\mathfrak{s}^d = s\cO_K$, then $\beta \coloneqq bs \in \cO_K$, and our first assertion follows. 

For the second assertion, we recall that by Dirichlet's $S$-Unit Theorem we have an isomorphism $\cO_S^{\times} \simeq \roots_K \times A$, where $A \simeq \Z^{| S | - 1}$.
Now, we let $r$ be the minimal positive integer for which  the ideal $\mathfrak{s}^r$ is principal. Clearly, every generator of $\mathfrak{s}^r$ lies in $\cO_S^{\times} \cap \cO_K$, and multiplying by a root of unity we can choose a generator $u$ that lies in $A$. The minimality of $r$ implies that $u$ is a nondivisible element of $A$, and therefore is a fundamental unit since $A \simeq \Z^{| S | - 1}$. 
\end{proof}

\begin{lemma}\label{L:A103} 
For a nonzero ideal $\mathfrak{a}$ of $\cO_S$ not containing $\mu = | \roots_K |$, the reduction map $\rho_{\mathfrak{a}} \colon  \cO_S \to \cO_S/\mathfrak{a}$ is injective on $\roots_K \subset \cO_S$, and consequently, $\mu$ divides $\phi(\mathfrak{a})$. Furthermore, given a nonzero $a \in \cO_S$ such that $V(a) \not\subset S \cup V(\mu)$, we have $\mu \notin a \cO_S$, hence the above assertions are true for $\mathfrak{a} = a\cO_S$. 
\end{lemma}
\begin{proof}
If $\rho_{\mathfrak{a}}$ is not injective on $\roots_K$, then there there exists $\zeta \in \roots_K \setminus \{ 1 \}$ such that $\zeta \equiv 1 \pmod{\mathfrak{a}}$. Then $\zeta$ is a root of the polynomial $f(x) \coloneqq (x^{\mu} - 1)/(x - 1)$, and therefore 
$$
\mu = f(1) \equiv f(\zeta) \pmod{\mathfrak{a}} \equiv 0 \pmod{\mathfrak{a}}. 
$$
Thus, $\mu \in \mathfrak{a}$, a contradiction, proving the first part of the lemma. The second part is obvious.     
\end{proof}

\subsection{Ray class fields}\label{SS:Ray}

Let $\mathfrak{b} \subset \cO_K$ be a nonzero ideal with the prime factorization 
\begin{equation}\label{E:A303} 
\mathfrak{b} = \mathfrak{p}_1^{n_1} \cdots \mathfrak{p}_r^{n_r}. 
\end{equation}
For $i = 1, \ldots , r$ we set $v_i = v_{\mathfrak{p}_i}$, $\mathcal{P}_i = \mathcal{P}_{v_i}$, and $V(\mathfrak{b})= \{v_1, \ldots , v_r\}$. (We note that for a nonzero $b \in \cO_K$, we have $V(b\cO_K) = V(b)$ in our previous notations.)  Next, we define the following open subgroup of the group of ideles $J = J_K$: 
$$
R(\mathfrak{b}) = \prod_{v \in V^K_{\mathbb{R}}} K_v^{> 0} \times \prod_{v \in V^K_{\infty} \setminus V^K_{\mathbb{R}}} K_v^{\times} \times \prod_{i = 1}^r \left( 1 + \mathcal{P}_i^{n_i}  \right) \times \prod_{v \in V^K_f \setminus V(\mathfrak{b})} \cO_v^{\times},
$$
where $V^K_{\mathbb{R}}$ denotes the set of real valuations $v$ of $K$, and $K_v^{> 0}$ is the group of positive elements in $K_v = \mathbb{R}$.  It is easy to see that for $J(\infty) \coloneqq 
\prod_{v \in V^K_{\infty}} K_v^{\times} \times \prod_{v \in V^K_f} \cO_v^{\times}$, the index $[J(\infty) : R(\mathfrak{b})]$ is finite, hence the index $[J(\infty) \, K^{\times} : R(\mathfrak{b}) \, K^{\times}]$ is also finite (where $K^{\times}$ is the group of principal ideles). On the other hand, the quotient $J/J(\infty) \, K^{\times}$ is isomorphic to the class group of $\cO_K$, hence finite. Thus, $R(\mathfrak{b}) \, K^{\times}$ is an open finite-index subgroup of $J$. By global class field theory, there exists a unique finite abelian extension $L/K$ whose norm group $N_{L/K}(J_L) \, K^{\times}$ coincides with $R(\mathfrak{b}) \, K^{\times}$. 

\begin{definition}
The finite abelian extension of $K$ with the norm group $R(\mathfrak{b}) \, K^{\times}$ is called the (narrow) \emph{ray class field} corresponding to $\mathfrak{b}$ and is denoted $\ray{\mathfrak{b}}$. 
\end{definition}

Then the reciprocity map $\alpha_{\ray{\mathfrak{b}}/K} \colon J \to \Gal \bigl( \ray{\mathfrak{b}}/K \bigr)$ is a continuous surjective group homomorphism with kernel $R(\mathfrak{b}) \, K^{\times}$ yielding thereby the fundamental isomorphism of global class field theory $J/ R(\mathfrak{b}) \, K^{\times} \simeq \Gal \bigl( \ray{\mathfrak{b}}/K \bigr)$. 

As in \cite[\S 3]{MRS}, for $c \in K^{\times}$ we set
$$
\theta_{\mathfrak{b}}(c) = \alpha_{\ray{\mathfrak{b}}/K}\bigr( \mathbf{j}_{\mathfrak{b}}(c) \bigr)^{-1} \in \Gal \bigl( \ray{\mathfrak{b}}/K \bigr)
$$
where $\mathbf{j}_{\mathfrak{b}}(c) \in J$ has the following components:
$$
\mathbf{j}_{\mathfrak{b}}(c)_v = 
    \begin{cases}
    \ c, & v \in V(\mathfrak{b}), \\ 
    \ 1 , & v \notin V(\mathfrak{b}). 
    \end{cases}
$$
Clearly, $\theta_{\mathfrak{b}} \colon K^{\times} \to \Gal \bigl( \ray{\mathfrak{b}}/K \bigr)$ is a group homomorphism.

\section{Some results from \texorpdfstring{\cite{MRS}}{Morgan-Rapinchuk-Sury}} \label{S:MRS}

The first result provides the existence of $\Q$-\emph{split} primes in a given residue class.

\begin{thm}[{cf. \cite[Theorem 3.3]{MRS}}] \label{MRS3.3}
Let $\mathfrak{b}$ be a nonzero ideal of $\cO_K$ and $a \in \cO_K$ be a nonzero element which is relatively prime to $\mathfrak{b}$ \textup(i.e., $a\cO_K + \mathfrak{b} = \cO_K$\textup). Then there exist infinitely many principal $\Q$-split prime ideals $\mathfrak{p}$ of $\cO_K$ that have a generator $\pi$ satisfying $\pi \equiv a \pmod{\mathfrak{b}}$ and $\pi > 0$ in all real completions of $K$. 
\end{thm}
 
The proof generally follows the strategy used in the proof of Dirichlet's Theorem in \cite[(A.12)]{BMS}, which we already used in the proof of Proposition \ref{P:A01}. The theorem was stated in \cite{MRS} for the ring $\cO_S$ where $S$ can be any finite subset of $V^K$ containing $V^K_{\infty}$, but we only need it for the ring of integers $\cO_K$, i.e., when $S = V^K_{\infty}$. It should also be noted that the argument in \cite{MRS} was given for the case where $\mathfrak{b}$ is principal but the general case does not require any changes. In fact, the general case can be reduced to the case of a principal $\mathfrak{b}$ by replacing a given $\mathfrak{b}$ with its suitable power which becomes principal. 

We will now quote the following statement which is needed for refinements of Theorem \ref{MRS3.3}.

\begin{prop}[{\cite[Proposition~3.5]{MRS}}] \label{MRS3.5}
Let $\mathfrak{b}$ be a nonzero ideal of $\cO_K$, let $a \in \cO_K$ be a nonzero element which is relatively prime to $\mathfrak{b}$, and let $F$ be a finite Galois extension of $\Q$ that contains the ray class field $\ray{\mathfrak{b}}$. Assume that a rational prime $p$ is unramified in $F$ and there exists an extension $w$ of the $p$-adic valuation $v_p$ to $F$ such that the Frobenius automorphism of $F/\Q$ corresponding to $w \vert v_p$ restricts to $\theta_{\mathfrak{b}}(a)$ on $\ray{\mathfrak{b}}$. If the restriction $v$ of $w$ to $K$ does not belong to $V(\mathfrak{b})$ then 
\begin{enumerate}[label=\upshape{(\alph*)}, ref=\alph*]
	\item \label{MRS3.5-a}
	$K_v = \Q_p$,
	and
	\item \label{MRS3.5-b}
	the prime ideal\/ $\mathfrak{p}_v$ of $\cO_K$ corresponding to~$v$ is principal with a generator~$\pi$ such that $\pi \equiv a \pmod{\mathfrak{b}}$ and $\pi > 0$ in every real completion of~$K$.
	\end{enumerate}
\end{prop}

We note that the above statement is the particular case of Proposition 3.5 in \cite{MRS} that is obtained by letting $S = V^K_{\infty}$. 

\vskip2mm 

Next, we will quote a statement pertaining to Galois theory that we will need in the  next section.

\begin{lemma}[cf.~{\cite[Lemma~3.9]{MRS}}]\label{MRS3.9}
Let $P$ be a finite set of rational primes, and set $\mu' = \mu \cdot \prod_{p \in P} p$. Furthermore, let $L_1$ be a finite abelian extension of $K$, and let $L_2$ be the splitting field over $K$ of the polynomial $x^{\mu'} - u$, where $u \in \cO_S^{\times}$ is a fundamental unit. Then, given any $\sigma_0 \in \Gal \bigl( (L_1 \cap L_2)/K \bigr)$, there exists $\sigma_2 \in \Gal(L_2/K)$ such that 
\begin{enumerate}
	\item \label{MRS3.9-restrict}
	$\sigma_2 |(L_1 \cap L_2) = \sigma_0$,
	and
	\item \label{MRS3.9-nontriv}
	for every $p \in P$, if $p^d = \mu_p$, then the field $L_2$ contains a root of $u$ of degree 
    $p^{d+1} = \mu_p^+$ that is not fixed by $\sigma_2$.
	\end{enumerate}
\end{lemma}

Finally, we will need the following technical statement. 

\begin{lemma}[{\cite[Lemma~4.3]{MRS}}] \label{MRS4.3}
Suppose we are given $(a,b)\in \row(\cO_S)$, and a finite subset $T\subseteq V^K_f$.  Then there exist $\alpha \in \cO_K$ and $r\in\cO_S^\times$ such that
$V(\alpha)\cap T=\varnothing$, and $(a,b)\reduce{1}(\alpha r,b)$.
\end{lemma}

\section{Proof of Theorem \ref{T:1} in characteristic zero} \label{S:NumberFieldPf}

\subsection{Main argument} It is enough to show that for $(a_0 , b_0) \in \row(\cO_S)$ we have $(a_0 , b_0) \reduce{6} (1 , 0)$. First, by Lemma \ref{MRS4.3} there exists $a \in \cO_S$ of the form $a = \alpha r$ with $\alpha \in \cO_K$ and $r \in \cO_S^{\times}$, such that $V(\alpha) \cap \bigl( S \cup V(\mu) \bigr) = \varnothing$ and $(a_0 , b_0) \reduce{1} (a , b)$, where $b = b_0$. So, it remains to show that if $(a , b) \in \row(\cO_S)$ is such that $a = \alpha r$ where $\alpha \in \cO_K$ and $r \in \cO_S^{\times}$ with $V(\alpha) \cap (S \cup V(\mu)) = \varnothing$, then $(a , b) \reduce{5} (1 , 0)$. The technical part of the argument is contained in the following number-theoretic statement. We let $u \in \cO_S^{\times}$ denote a fundamental unit that belongs to $\cO_K$ -- see Lemma  \ref{L:A102}.

\begin{thm}\label{T:A301}
Let $(a , b) \in \row(O_S)$ be such that $a \notin \cO_S^{\times}$ and for some unit $r \in \cO_S^{\times}$ we have $\alpha \coloneqq a r \in \cO_K$ and $V(\alpha) \cap \bigl( S \cup V(\mu) \bigr) = \varnothing$. Furthermore, let $P \subset V^K$ be a finite set containing $S$. Then there exist $\Q$-split principal ideals $\mathfrak{p} ,  \mathfrak{q}$ of $\cO_K$ with generators $\pi$ and $q$ respectively such that $v_{\mathfrak{p}} , v_{\mathfrak{q}} \notin P$ and 
\begin{enumerate}
    \item \label{T:A301-q=b}
    $q \equiv b \pmod{a\cO_S}$, 

\item \label{T:A301-q=power}
$q \pmod{\mathfrak{p}}$ is a $\mu$-th power in $(\cO_K/\mathfrak{p})^{\times}$, 

\item \label{T:A301-pigen}
$\pi \pmod{\mathfrak{q}^2}$ generates $(\cO_K/\mathfrak{q}^2)^{\times}$, 

\item \label{T:A301-mu}
$\mu \mid \phi(\mathfrak{p})$ and for each prime $p$ dividing $\phi(a \cO_S)$, the $p$-primary component of $\phi(\mathfrak{p})/\mu$ divides the $p$-primary component of the order of $u \pmod{\mathfrak{p}}$ in $(\cO_K/\mathfrak{p})^{\times}$.
\end{enumerate}
\end{thm}

We will now assume Theorem \ref{T:A301} and complete the proof of Theorem \ref{T:1}; the proof of Theorem \ref{T:A301} will occupy the remainder of the article. So, let $(a , b) \in \row(\cO_S)$ be such that for some $r \in \cO_S^{\times}$ we have $\alpha \coloneqq ar \in \cO_K$ with $V(\alpha) \cap \bigl( S \cup V(\mu) \bigr) = \varnothing$. In addition we can assume that $a \notin \cO_S^{\times}$ as otherwise $(a , b) \reduce{3} (1 , 0)$. Applying Theorem \ref{T:A301} with $P = V \bigl( \phi(a \cO_S) \bigr) \cup S \cup V(\mu)$, we find prime ideals $\mathfrak{p} , \mathfrak{q} \notin P$ having generators $\pi$ and $q$ that satisfy properties (1)-(4) listed in the theorem. Let $\ell = \ell_{\!\mathfrak{p}}(u)$ be the $\mathfrak{p}$-level of $u$ 
(cf.\ Definition \ref{D:level}). 

\vskip2mm 

We have $a \cO_S = \alpha \cO_S$. Since $V(\alpha) \cap V(\mu) = \varnothing$, 
we conclude that $\mu \notin a\cO_S$, and therefore $\mu \mid \phi(a\cO_S)$ by Lemma \ref{L:A103}. At the same time, $\mu \mid \phi(\mathfrak{p})$. So, $\mu$ divides  $\gcd \bigl( \phi(a\cO_S) , \phi(\mathfrak{p}) \bigr)$, which in turn divides $\phi_0 =$ product of $p$-primary parts  $\phi(\mathfrak{p})_p$ for all primes dividing $\phi(a\cO_S)$. Besides, we clearly have $\gcd \bigl( \phi(\mathfrak{p})/\phi_0 , \phi(a\cO_S) \bigr) = 1$. Furthermore, it is well known that 
$$
\phi(\mathfrak{p}^{\ell}) = \phi(\mathfrak{p}) \cdot | \cO_K/\mathfrak{p} |^{\ell-1},
$$
so since $\mathfrak{p} \notin V \bigl( \phi(a\cO_S) \bigr)$, 
$$
\gcd \bigl( \phi(\mathfrak{p}^{\ell})/\phi_0 , \phi(a\cO_S) \bigr) = \gcd \bigl( \phi(\mathfrak{p})/\phi_0 , \phi(a\cO_S) \bigr) = 1. 
$$
Thus, we can pick an integer $m \geq 2$ satisfying 
$$
m \equiv 0 \pmod{\phi(\mathfrak{p}^{\ell})/\phi_0} \ \ \text{and} \ \ m \equiv 1 \pmod{\phi(a\cO_S)}. 
$$
Now:

\medbreak 

$\bullet$ \emph{$(a , b) \reduce{2} (\pi^t , q^m)$ for some $t \geq \ell$.} Indeed, we have 
$$
q^m \equiv q \equiv b \pmod{a\cO_S}, 
$$
so $(a , b) \reduce{1} (a , q^m)$. Furthermore, the fact that  $\pi \pmod{\mathfrak{q}^2}$ generates $(\cO_K/\mathfrak{q}^2)^{\times}$ implies that $\pi \pmod{\mathfrak{q}^m}$ generates $(\cO_K/\mathfrak{q}^m)^{\times}$ -- cf. Lemma \ref{MRS3.1}.  Thus, we can find an integer 
$t \geq \ell$ such that $(a , q^m) \reduce{1} (\pi^t , q^m)$, and the required fact follows. (We note that this argument relied only on properties (\ref{T:A301-q=b}) and~(\ref{T:A301-pigen}) from Theorem \ref{T:A301}.)

\medbreak 

$\bullet$ \emph{$(\pi^t , q^m) \reduce{1} (\pi^t , u^n)$ for some $n$.} For this we apply Lemma \ref{L:A101} with $\lambda = \phi_0$, $\delta = \mu$ and $b = q$. By construction we have $m \equiv 0 \pmod{\phi(\mathfrak{p}^{\ell})/\lambda}$. Now, every prime divisor $p$ of $\lambda/\delta$ divides $\phi(a\cO_S)$, and therefore the $p$-primary component $(\phi(\mathfrak{p})/\mu)_p$ divides the $p$-primary component $d_p$ of the order $d$ of $u \pmod{\mathfrak{p}}$ in $(\cO_K/\mathfrak{p})^{\times}$. Since $(\lambda/\delta)_p$ divides $(\phi(\mathfrak{p})/\mu)_p$, we obtain that $\lambda/\delta$ divides $d$. This verifies all the assumptions of Lemma \ref{L:A101}, and its application yields the required fact. (This argument used only properties (\ref{T:A301-q=power}) and~(\ref{T:A301-mu}) from Theorem \ref{T:A301}.) 

\medbreak 

Since $u \in \cO_S^{\times}$, it is easy to see that $(\pi^t , u^n) \reduce{2} (1 , 0)$, completing the proof of Theorem \ref{T:1}.\qed

\subsection{Proof of Theorem \ref{T:A301}}

It follows from Hensel's Lemma (or the Inverse Function Theorem) that for every $v \in V^K$, 
the map $K_v^{\times} \to K_v^{\times}$, $x \mapsto x^{\mu}$, is open. So, there exists an integer $N > 0$ such that 
\begin{equation}\label{E:power}
    1 + \mathcal{P}_v^N \subset ({K_v^{\times}})^{\mu} \ \ \text{for all} \ \ v \in V(\mu), 
\end{equation}
and we define the following ideal of $\cO_K$: 
$$
\mathfrak{n} = \prod_{v \in V(\mu)} \mathfrak{p}_v^N, 
$$
where $\mathfrak{p}_v \subset \cO_K$ is the prime ideal corresponding to $v \in V^K_f$. Next, let 
$$
\mu = p_1^{e_1} \cdots p_d^{e_d}
$$
be the prime factorization of $\mu$, and for each $i = 1, \ldots , d$ fix one $v_i \in V(p_i)$. It follows from property (\ref{SS:PRS-proots}) in subsection \ref{SS:PRS} that there exist $u_i , u'_i \in \cO_{v_i}^{\times}$ such that  $\powres{u_i , u'_i}{v_i}{p_i^{e_i}}$ is a primitive $p_i^{e_i}$-th root of unity. We claim that $\zeta_i \coloneqq \powres{u_i , u'_i}{v_i}{\mu}$ is also a primitive $p_i^{e_i}$-th root of unity. Indeed, we see from properties (\ref{SS:PRS-takepwer}) and (\ref{powres-cong}) in \ref{SS:PRS} that
$$
\zeta_i^{p_i^{e_i}} = \powres{u_i , u'_i}{v_i}{\mu/p_i^{e_i}} = 1 ,
$$
i.e., $\zeta_i$ is a $p_i^{e_i}$-th root of unity. On the other hand, 
$$
\zeta_i^{\mu/p_i^{e_i}} = \powres{u_i , u'_i}{v_i}{p_i^{e_i}}
$$
is a primitive $p_i^{e_i}$-th root of unity, so $\zeta_i$ itself is a primitive $p_i^{e_i}$-th root of unity. It follows that 
$$
\zeta \coloneqq \prod_{i = 1}^d \zeta_i
$$
is a primitive $\mu$-th root of unity. 

Next, we can find elements $z_0 , z'_0 \in \cO_K$ such that for each $i = 1, \ldots , d$ we have 
$$
z_0 \equiv u_i \pmod{\mathcal{P}_{v_i}^N} \ \ \text{and} \ \ z_0 \equiv 1 \pmod{\mathcal{P}_v^N} \ \ \text{for all} \ \ v \in V(p_i) \setminus \{ v_i \}, 
$$
and similarly,
$$
z'_0 \equiv u'_i \pmod{\mathcal{P}_{v_i}^N} \ \ \text{and} \ \ z'_0 \equiv 1 \pmod{\mathcal{P}_v^N} \ \ \text{for all} \ \ v \in V(p_i) \setminus \{ v_i \}. 
$$
Then taking (\ref{E:power}) into account, we obtain 
$$
\prod_{v \in V(\mu)} \powres{z_0 , z'_0}{v}{\mu} = \prod_{i = 1}^d   \powres{z_0 , z'_0}{v_i}{\mu} = \prod_{i = 1}^d   \powres{u_i , u'_i}{v_i}{\mu} = \zeta. 
$$
Now, given an integer $j > 0$ and $z , z' \in \cO_K$ such that $z \equiv z_0^j\pmod{\mathfrak{n}}$ and $z' \equiv z'_0 \pmod{\mathfrak{n}}$ then it follows from (\ref{E:power}) that $z/z_0^j \,, \, z'/z'_0 \in (K_v^{\times})^{\mu}$ for all $v \in V(\mu)$. So, we obtain the following: 
\begin{equation}\label{E:A501}
\begin{matrix}
\text{for any $z , z' \in \cO_K$ such that $z \equiv z_0^j \pmod{\mathfrak{n}}$ and $z' \equiv z'_0 \pmod{\mathfrak{n}}$ we have}
\\ \displaystyle
\prod_{v \in V(\mu)} \powres{z , z'}{v}{\mu} = \prod_{v \in V(\mu)} \powres{z_0^j , z'_0}{v}{\mu} = \zeta^j. 
\end{matrix}
\end{equation}    

\refstepcounter{stepholder} 
\begin{step} \label{T:A301Pf-q}
Construction of $\mathfrak{q}$ and $q$.%
\end{step}
By Lemma \ref{L:A102} we can write $b = \beta s^{-1}$ with $\beta \in \cO_K$ and $s \in \cO_S^{\times} \cap \cO_K$. As $V(\alpha) \cap S = \varnothing,$ there exists $t \in \cO_K$ such that $st \equiv 1 \pmod{\alpha \cO_K}$. Since $V(\alpha) \cap V(\mu) = \varnothing$, and hence $\alpha$ and $\mathfrak{n}$ are relatively prime, it follows from Theorem \ref{MRS3.3}
that there exist infinitely many $\Q$-split prime ideals $\mathfrak{q}$ of $\cO_K$ having a generator $q$ satisfying
\begin{enumerate}
    \item $q > 0$ in every real completion of $K$; 
    \item $q \equiv \beta t \pmod{\alpha\cO_K}$;
    \item $q \equiv z'_0 \pmod{\mathfrak{n}}$. 
\end{enumerate}
(We note that since $a$ and $b$ are relatively prime in $\cO_S$, we have $V(\alpha) \cap V(\beta) \subset S$, and hence $V(\alpha) \cap V(\beta) = \varnothing$. But the ideal of $\cO_K$ generated by $\alpha$ and $\beta t$ contains $\alpha$ and $\beta$, implying that $\alpha$ and $\beta t$ are relatively prime in $\cO_K$. By the Chinese Remainder Theorem, we can choose $c \in \cO_K$ so that $c \equiv \beta t \pmod{\alpha \cO_K}$ and $c \equiv z'_0 \pmod{\mathfrak{n}}$, observing that $c$ is prime to $\mathfrak{b} \coloneqq \alpha \mathfrak{n}$. Then we apply Theorem \ref{MRS3.3} to $c$ and $\mathfrak{b}$.)  In particular, we can (and do) fix such a $\mathfrak{q}$ so that $v_{\mathfrak{q}} \notin P \cup S \cup V(\mu) \cup V(\phi(\alpha \cO_K))$. We observe that 
$$
qs \equiv \beta \equiv bs \pmod{\alpha \cO_K}, 
$$
so $q \equiv b \pmod{a\cO_S}$. 
This verifies property~(\ref{T:A301-q=b}) of the statement of the theorem.

\begin{step} \label{T:A301Pf-p}
Construction of\/ $\mathfrak{p}$ and $\pi$.%
\end{step}
Let $\mathfrak{b} = \mathfrak{q}^2 \mathfrak{n}$, where $\mathfrak{q}$ is the prime ideal of $\cO_K$ fixed in Step \ref{T:A301Pf-q}, and let $L_1 = \ray{\mathfrak{b}}$ be the corresponding ray class field (cf.\ \ref{SS:Ray}). Furthermore, let 
$$
\mu' = \mu \cdot \prod_{p \in \Phi} p 
$$
where $\Phi$ is the set of prime divisors of $\phi(a \cO_S)$, and let $L_2$ be the splitting field (over $K$) of the polynomial $x^{\mu'} - u$. 

We will now define an automorphism $\sigma_1 \in \Gal(L_1/K)$. By Lemma \ref{MRS3.1}, the group $(\cO_K/\mathfrak{q}^2)^{\times}$ is cyclic, so we can choose $u_{\mathfrak{q}} \in \cO_K$ for which $u_{\mathfrak{q}} \pmod{\mathfrak{q}^2}$ generates $(\cO_K/\mathfrak{q}^2)^{\times}$. Next, since $\zeta$ is a primitive $\mu$-th root of unity, we can choose an integer $j > 0$ such that 
$$
\zeta^j = \powres{u_{\mathfrak{q}} , q}{\mathfrak{q}}{\mu}^{-1}. 
$$
We then use the Chinese Remainder Theorem to find $c \in \cO_K$ satisfying 
$$
c \equiv u_{\mathfrak{q}} \pmod{\mathfrak{q}^2} \ \ \text{and} \ \ c \equiv z_0^j \pmod{\mathfrak{n}}. 
$$
Since $u_{\mathfrak{q}}$ and $z_0^j$ are units modulo $\mathfrak{q}^2$ and $\mathfrak{n}$ respectively, we conclude that $c$ is relatively prime to $\mathfrak{b}$, and we can consider 
$$
\sigma_1 \coloneqq \theta_{\mathfrak{b}}(c) \in \Gal(L_1/K)
$$
in the notations of \ref{SS:Ray}. 

Set $\sigma_0 = \sigma_1 | (L_1 \cap L_2)$. Using Lemma \ref{MRS3.9}, we can find an automorphism $\sigma_2 \in \Gal(L_2/K)$ such that $\sigma_2 | (L_1 \cap L_2) = \sigma_0$ and for every prime $p \in \Phi$, the field $L_2$ contains a root of $u$ of degree $\mu_p^+$ that is not fixed by $\sigma_2$. By construction, 
$$
\sigma_1 | (L_1 \cap L_2) = \sigma_2 | (L_1 \cap L_2),
$$
so there exists $\sigma \in \Gal(L_1L_2/K)$ such that $\sigma | L_i = \sigma_i$ for $i = 1, 2$. 

Let $F$ be the Galois closure of $L_1L_2$ over $\Q$, and let $\widetilde{\sigma} \in \Gal(F/\Q)$ be such that $\widetilde{\sigma} | (L_1L_2) = \sigma$. Using Chebotarev's Density Theorem, we can find a rational prime $p > 2$ which is unramified in $F$ and such that no extension of the $p$-adic valuation $v_p$ to $K$ lies in $P \cup V(\mathfrak{b}) \cup \Phi$ and for a suitable extension $w$ of $v_p$ to $F$ the corresponding Frobenius automorphism $\varphi$ coincides with $\widetilde{\sigma}$. Let $v$ be the restriction of $w$ to $K$, and let $\mathfrak{p}$  be the prime ideal of $\cO_K$ corresponding to $v$.  Since 
$$
\varphi | L_1 = \sigma_1 = \theta_{\mathfrak{b}}(c), 
$$
we conclude from Proposition \ref{MRS3.5} that $K_v = \Q_p$, implying the ideal $\mathfrak{p}$ is $\Q$-split, and furthermore that $\mathfrak{p}$ is principal with a generator $\pi$ satisfying 
$$
\pi \equiv c \pmod{\mathfrak{b}} \ \ \text{and} \ \ \pi > 0 \ \ \text{in every real completion of} \ \ K. 
$$
In particular, $\pi \equiv u_{\mathfrak{q}}\pmod{\mathfrak{q}^2}$, and hence $\pi \pmod{\mathfrak{q}^2}$ generates $(\cO_K/\mathfrak{q}^2)^{\times}$. 
This verifies property~(\ref{T:A301-pigen}) of the statement of the theorem.

\vskip2mm 

Properties (\ref{T:A301-q=b}) and (\ref{T:A301-pigen}) in Theorem \ref{T:A301} have already been established in the course of implementing the constructions in Steps \ref{T:A301Pf-q} and~\ref{T:A301Pf-p}. So, it remains to verify properties (\ref{T:A301-q=power}) and~(\ref{T:A301-mu}).  

\begin{step}
Verification of property (\ref{T:A301-q=power}) in Theorem \ref{T:A301}.%
\end{step}
By using property (\ref{powres-power}) in \ref{SS:PRS} for the archimedean places (and the fact that $\pi , q > 0$ in all real completions of $K$) and using property (\ref{powres-cong}) in \ref{SS:PRS} for the nonarchimedean places, we see that 
$$
\powres{\pi , q}{v}{\mu} = 1 \ \ \text{for all} \ \ v \in V^K \setminus (\{v_{\mathfrak{p}} , v_{\mathfrak{q}} \} \cup V(\mu)).
$$
So, it follows from the Reciprocity Law (property (\ref{powres-reciprocity}) in \ref{SS:PRS}) that 
\begin{equation}\label{E:A601}
\powres{\pi , q}{\mathfrak{p}}{\mu} \cdot \powres{\pi , q}{\mathfrak{q}}{\mu} \cdot \prod_{v \in V(\mu)} \powres{\pi , q}{v}{\mu} = 1. 
\end{equation}
By construction, we have $\pi \equiv z_0^j \pmod{\mathfrak{n}}$ and $q \equiv z'_0 \pmod{\mathfrak{n}}$, so applying (\ref{E:A501}) to $z = \pi$ and $z' = q$ we obtain that 
$$
\prod_{v \in V(\mu)} \powres{\pi , q}{v}{\mu} = \zeta^j = \powres{u_{\mathfrak{q}} , q}{\mathfrak{q}}{\mu}^{-1}. 
$$
Also, since $\pi \equiv u_{\mathfrak{q}}\pmod{\mathfrak{q}}$, we see from property (\ref{powres-cong}) in \ref{SS:PRS} that 
    $$ \powres{\pi , q}{\mathfrak{q}}{\mu}
    = \powres{u_{\mathfrak{q}} , q}{\mathfrak{q}}{\mu} . $$
This means that (\ref{E:A601}) reduces to 
$$
\powres{\pi , q}{\mathfrak{p}}{\mu} = 1. 
$$
By property (\ref{powres-mpower}) in \ref{SS:PRS}, this means that $q$ is a $\mu$-th power modulo $\mathfrak{p}$, as required. 

\begin{step}
Verification of property (\ref{T:A301-mu}) in Theorem \ref{T:A301}.%
\end{step}
Let $v = v_{\mathfrak{p}}$.
Since $v \notin V(\mu)$, hence $\mu \notin \mathfrak{p}$, the fact that $\mu$ divides $\phi(\mathfrak{p})$ follows from Lemma \ref{L:A103}. Let $t$ be the order of $\bar{u} = u \pmod{\mathfrak{p}}$ in $(\cO_K/\mathfrak{p})^{\times}$. We need to show that for every $\ell \in \Phi$ we have 
$$
\bigl( \phi(\mathfrak{p})/\mu \bigr)_{\ell} \le t_{\ell}. 
$$
Set $\ell^d = \mu_{\ell}$. If $\bigl( \phi(\mathfrak{p})/\mu \bigr)_{\ell} > t_{\ell}$ then $\phi(\mathfrak{p})/\ell^{d+1}$ is an integer divisible by $t_{\ell}$. Since the group $(\cO_K/\mathfrak{p})^{\times}$ is cyclic, this means that the $\ell$-primary component $\langle \bar{u} \rangle_{\ell}$ of the cyclic group $\langle \bar{u} \rangle$ is contained in the subgroup of $\ell^{d+1}$-th powers in $(\cO_K/\mathfrak{p})^{\times}$. Then the whole cyclic group $\langle \bar{u} \rangle$ is contained in this subgroup, and in particular, $\bar{u}$ itself is a $\ell^{d+1}$-th power. Moreover, since $\ell^{d+1} \mid \phi(\mathfrak{p})$, we have $\ell^{d+1}$ roots of unity of degree $\ell^{d+1}$ in $\cO_K/\mathfrak{p}$. All this means that the polynomial $\bar{f}(x) = x^{\ell^{d+1}} - \bar{u}$, which is the reduction of the polynomial $f(x) = x^{\ell^{d+1}} - u$, is a product of linear factors over $\cO_K/\mathfrak{p}$. Since $\ell \notin \mathfrak{p}$, the polynomial $\bar{f}$ is separable, so it follows from Hensel's Lemma that $f$ is a product of linear factors over $K_v$. Since $\sigma_2 \in \Gal \bigl( (L_2)_v/K_v \bigr)$ (it is actually the Frobenius automorphism of the extension $(L_2)_v/K_v$), we obtain that $\sigma_2$ acts trivially on all roots of $f$, which contradicts the construction of~$\sigma_2$.

\vskip1mm

This concludes the proof of Theorem \ref{T:A301}. \qed

\subsection*{AI Declaration}
AI was not used in the preparation of this manuscript.

\bibliographystyle{amsplain}

\end{document}